\documentclass{article}

\usepackage{amsmath,amsfonts,bm}

\def\eqref#1{equation~(\ref{#1})}

\def\1{\bm{1}}

\def\rr{{\textnormal{r}}}
\def\rs{{\textnormal{s}}}

\def\rx{{\textnormal{x}}}

\def\rz{{\textnormal{z}}}

\def\rvg{{\mathbf{g}}}

\def\vzeta{{\bm{\zeta}}}

\def\ve{{\bm{e}}}

\def\vg{{\bm{g}}}

\def\vs{{\bm{s}}}

\def\vu{{\bm{u}}}
\def\vv{{\bm{v}}}

\def\vx{{\bm{x}}}
\def\vy{{\bm{y}}}

\def\mA{{\bm{A}}}

\def\mG{{\bm{G}}}
\def\mH{{\bm{H}}}
\def\mI{{\bm{I}}}

\def\mM{{\bm{M}}}
\def\mN{{\bm{N}}}

\def\mU{{\bm{U}}}
\def\mV{{\bm{V}}}

\def\mSigma{{\bm{\Sigma}}}

\DeclareMathAlphabet{\mathsfit}{\encodingdefault}{\sfdefault}{m}{sl}
\SetMathAlphabet{\mathsfit}{bold}{\encodingdefault}{\sfdefault}{bx}{n}

\def\sB{{\mathbb{B}}}

\def\sI{{\mathbb{I}}}

\def\sR{{\mathbb{R}}}

\newcommand{\grad}{\ensuremath{\nabla}}
\newcommand*\circled[1]{\tikz[baseline=(char.base)]{
            \node[shape=circle,draw,inner sep=1pt] (char) {\tiny{#1}};}}

\makeatletter
\def\thickhline{%
  \noalign{\ifnum0=`}\fi\hrule \@height \thickarrayrulewidth \futurelet
   \reserved@a\@xthickhline}
\def\@xthickhline{\ifx\reserved@a\thickhline
               \vskip\doublerulesep
               \vskip-\thickarrayrulewidth
             \fi
      \ifnum0=`{\fi}}
\makeatother

\newlength{\thickarrayrulewidth}
\newcommand{\tabitem}{~~\llap{\textbullet}~~}

\usepackage{arxiv}
\usepackage{amsmath}
\usepackage{amssymb}
\usepackage[utf8]{inputenc} 
\usepackage[T1]{fontenc}    
\usepackage{hyperref}       
\usepackage{url}            
\usepackage{booktabs}       
\usepackage{amsfonts}       
\usepackage{nicefrac}       
\usepackage{microtype}      
\usepackage{lipsum}
\usepackage{xspace}
\usepackage{breqn}
\usepackage[ruled,vlined]{algorithm2e}
\usepackage{graphicx}
\usepackage{amsthm}
\usepackage{mathrsfs}
\usepackage{mathtools}
\usepackage{tikz}
\usepackage{multicol}
\usepackage{multirow}

\newtheorem{lemma}{Lemma}[section]
\newtheorem{theorem}[lemma]{Theorem}
\newtheorem{corollary}[lemma]{Corollary}

\newtheorem*{remark*}{Remark}
\newtheorem{assumption}{Assumption}
\newtheorem{definition}{Definition}

\makeatletter
\newcommand\mathcircled[1]{%
  \mathpalette\@mathcircled{#1}%
}
\newcommand\@mathcircled[2]{%
  \tikz[baseline=(math.base)] \node[draw,circle,inner sep=1pt] (math) {$\m@th#1#2$};%
}

\newcommand{\radagrad}{SHAdaGrad\xspace}

\title{Adaptive Gradient Methods Can Be Provably Faster than SGD after Finite Epochs}

\author{
  Xunpeng Huang\\
  Bytedance AI Lab\\
  \texttt{huangxunpeng@bytedance.com}
    \And
  Hao Zhou\\
  Bytedance AI Lab\\
  \texttt{zhouhao.nlp@bytedance.com}
    \And
  Runxin Xu\\
  Peking University\\
  \texttt{runxinxu@gmail.com}
    \And
  Zhe Wang\\
  Ohio State University\\
  \texttt{wang.10982@osu.edu}
    \And
  Lei Li\thanks{Corresponding author}\\
  Bytedance AI Lab\\
  Beijing, China\\
  \texttt{lileilab@bytedance.com}
}

\begin{document}
\maketitle

\begin{abstract}
Adaptive gradient methods have attracted much attention of machine learning communities due to the high efficiency.
However their acceleration effect in practice, especially in neural network training, is hard to analyze, theoretically.
The huge gap between  theoretical convergence results and practical performances prevents further understanding of existing optimizers and the development of more advanced optimization methods.
In this paper, we provide adaptive gradient methods a novel analysis with an additional mild assumption, and revise AdaGrad to \radagrad for matching a better provable convergence rate.
To find an $\epsilon$-approximate first-order stationary point in non-convex objectives, we prove random shuffling \radagrad achieves a $\tilde{O}(T^{-1/2})$ convergence rate, which is significantly improved by factors $\tilde{O}(T^{-1/4})$ and $\tilde{O}(T^{-1/6})$ compared with existing adaptive gradient methods and random shuffling SGD, respectively.
To the best of our knowledge, it is the first time to demonstrate that adaptive gradient methods can deterministically be faster than SGD after finite epochs.
Furthermore, we conduct comprehensive experiments to validate the additional mild assumption and the acceleration effect benefited from second moments and random shuffling.
\end{abstract}

\section{Introduction}
\label{sec:intro}
Stochastic optimization is critical for large scale machine learning, formally, which aims to solve the following finite sum minimization problem:
\begin{equation}
    \label{eq:obj_intro}
    \min_{\vx \in \sR^d}\quad f(\vx) = \frac{1}{n}\sum_{i=1}^n f_i(\vx),
\end{equation}
where each function $f_i:\mathbb{R}^d\ \rightarrow\ \mathbb{R}$  is smooth and possibly non-convex.
This problem covers a wide range of models in machine learning, including deep neural networks (DNNs).
When training DNNs, adaptive gradient methods~\cite{duchi2011adaptive,kingma2014adam,
reddi2018on} are usually much faster than stochastic gradient descent (SGD) in practice, while their theoretical convergence is the same as or even worse than SGD in the non-convex setting~\cite{chen2018on,zhou2018convergence,zaheer2018adaptive,ward2018adagrad}.
This inconsistency between practical performances and theoretical convergence prevents further understanding of existing optimizers and the development of more advanced optimization methods.
Thus, closing the gap between practical performances and theoretical results is a very important issue. 

Various studies attempt to bridge such a gap from different perspectives.
Although these previous work offer very promising insights, they hardly explain why adaptive gradient methods can be faster than SGD theoretically.
For example, \cite{reddi2018on} corrects errors in regret convergence analysis, \cite{chen2018sadagrad,wang2019sadam} provide the convergence analysis for achieving the global optimum in strongly convex optimization, and \cite{chen2018on,zhou2018convergence,zaheer2018adaptive,ward2018adagrad} investigate the convergence rate for achieving first-order stationary points (FSPs) in non-convex setting to match assumptions in practice.
All these studies can only provide similar convergence results to SGD, i.e., $O(T^{-1/2})$ for regret, $O(T^{-1})$ for achieving global optimum and $\tilde{O}(T^{-1/4})$ for achieving FSPs.

We argue that previous studies ignore the effects of random shuffling (sampling without replacement) in their analysis, which could accelerate SGD according to~\cite{haochen2018random} and ~\cite{nguyen2020unified}. 
Furthermore, the sampling strategy of mini-batch gradient in each epoch is random shuffling rather than uniformed sampling as previous work assumed when DNNs are trained in practice. 
With some indirect evidences, e.g., the geometric properties of second moments proposed in~\cite{huang2020acutum}, we suspect that adaptive gradient methods may have better adaptability  compared with SGD.
Thus, we target at filling the gap between practical performances and theoretical results for adaptive gradient methods via investigating the convergence rate for achieving FSPs in non-convex and \emph{random shuffling} settings.

In this paper, we show that, with an additional mild assumption, adaptive gradient methods can obtain an $\tilde{O}(T^{-1/2})$ convergence rate which outperforms previous best-known results.
Specifically, it is $O(T^{-1/3})$ for random shuffling SGD (\cite{haochen2018random,nguyen2020unified}), $O(T^{-1/4})$ for vanilla SGD (\cite{allen2018natasha}) and $\tilde{O}(T^{-1/4})$ for Adam-type optimizers (\cite{chen2018on,zhou2018convergence,zaheer2018adaptive}).
From a theoretical point of view, we explain improvement from two observations.
First, the combination of some full gradient perturbations and the second moment matrices can provide tighter lower bounds for sufficient descents in the random shuffling setting.
Second, tighter sufficient descent lower bounds improve the convergence rate by weakening sufficient conditions required for the convergence.
In practice, we first revise AdaGrad with full matrices~\cite{duchi2011adaptive} (AdaGrad\_F) to \radagrad for theoretical proving convenience.
Then, we conduct comprehensive experiments  to convince readers the mild assumption in our proof, and validate the acceleration effect from the introduction of second moments and random shuffling.
To the best of our knowledge, it is the first time to explain adaptive gradient methods can be deterministically faster than SGD after finite epochs both in theory and in practice.
The main contributions of this paper are as follows:
\begin{itemize}
    \item We are the first to analyze the convergence rate of adaptive gradient methods for achieving FSPs in non-convex and random shuffling settings, and provide an $\tilde{O}(T^{-1/2})$ convergence rate to \radagrad, a minor revision of AdaGrad\_F, with an additional mild assumption. 
    \item We conduct comprehensive experiments to validate our mild assumption, and present the acceleration effect taken from random shuffling and second moments.
\end{itemize}

\section{Related Work}
\label{sec:related}
In this section, we only introduce the work highly related to the analysis of adaptive gradient methods and  the random shuffling strategy due to the space limitation. 
We briefly describe the difference between the existing work and ours, and list all of the convergence results for comparison.

\textbf{Analysis for adaptive gradient methods}
Compared with classic optimization methods for non-convex objectives, e.g., SGD~\cite{ge2015escaping}, SVRG~\cite{lei2017non,allen2018natasha} and SPIDER~\cite{fang2018spider,wang2019spiderboost}, adaptive gradient methods, e.g., Adagrad~\cite{duchi2011adaptive}, Adam~\cite{kingma2014adam} and AMSGrad~\cite{reddi2018on}, are more popular due to their excellent practical performances for neural network training.
These adaptive gradient methods are originally proposed to solve online learning problems, and focus on the convergence analysis of their regret for convex objective functions.
To further understand online learning optimizers in neural network training, 
the convergence analysis for non-convex problems are highly desired.
Therefore, \cite{chen2018on}, \cite{zhou2018convergence}, \cite{zaheer2018adaptive} and \cite{ward2018adagrad} analyze the convergence rate for achieving first-order stationary points (FSPs).
Besides, they proposed a series of novel methods for faster convergence and better generalization.
However, the convergence results of the proposed methods are usually $\tilde{O}(T^{-1/4})$, which not better than the vanilla SGD in non-convex settings. 

\textbf{Analysis for random shuffling in optimization.}
In neural network training, instances are usually sent to optimizers after random shuffling.
With such a pre-processing, random shuffling is considered to be an important ingredient to capture the practical performance of optimization methods in theoretical analysis. 
Furthermore, the convergence of random shuffling SGD and vanilla SGD is quite different.
Compared with the uniform sampling for calculated gradient at each iteration in vanilla SGD, \cite{haochen2018random,nagaraj2019sgd} and \cite{nguyen2020unified} have fully explained advantages of random shuffling utilization in the convergence rate.
They improve the convergence rate from $O(T^{-1})$ and $O(T^{-1/4})$ to $O(T^{-2})$ and $\tilde{O}(T^{-1/3})$ for achieving FSPs in strongly convex and non-convex settings, respectively.

From related work, one may notice that the convergence results of adaptive gradient methods in non-convex and random shuffling settings are still understudied.
Within an additional mild assumption, we improve the order of the convergence rate by a factor $\tilde{O}(T^{-1/4})$ compared with vanilla SGD and existing Adam-type optimizers, and $\tilde{O}(T^{-1/6})$ compared with random shuffling SGD.
(See Table~\ref{tab:convergence_results_comp} for the details comparison)
\begin{table}[t]
    \caption{\small Convergence rate comparison of SGD, Adam-type optimizers, random shuffling SGD and \radagrad, where $T$ denotes the number of epoch, and $n$, i.e., the number of instances is considered as a constant.\label{tab:convergence_results_comp}}
    \centering
    \begin{tabular}{|l|l|l|c|}
        \thickhline
            \multicolumn{2}{|c|}{Algorithm} & Assumptions (L-smoothness+) & Convergence Results\\ 
        \hline
            \multicolumn{2}{|l|}{
                \begin{tabular}{@{}c@{}}
                vanilla SGD
                \end{tabular}} & 
            \begin{tabular}{@{}l@{}}
                \tabitem $\sigma^2$ bounded variance \\ 
            \end{tabular} & 
            \begin{tabular}{@{}l@{}}
                $\mathbb{E}\left[\left\|\rvg_t\right\|^2\right] = O\left(\frac{1}{\sqrt{T}}\right)$ \\
            \end{tabular} \\
        \hline
            \multirow{4}{*}{
                \rotatebox{90}{\begin{tabular}{@{}l@{}}
                    Adaptive Gradient\\ Methods Analysis
                \end{tabular}}
            } & 
            \begin{tabular}{@{}l@{}}
                AMSGrad,\\ AdaFom~\cite{chen2018on} 
            \end{tabular} & 
            \begin{tabular}{@{}l@{}}
                \tabitem bounded gradients\\ 
                \tabitem initial gradient coordinate \\ lower bound
            \end{tabular} & 
            \begin{tabular}{@{}l@{}}
                $\min\ \mathbb{E}\left[\left\|\rvg_t\right\|^2\right] = O\left(\frac{\ln T + d^2}{\sqrt{T}}\right)$
            \end{tabular} \\
        \cline{2-4}
            &
            \begin{tabular}{@{}l@{}}
                AMSGrad,\\
                Padam~\cite{zhou2018convergence} 
            \end{tabular} & 
            \begin{tabular}{@{}l@{}}
                \tabitem bounded gradients\\ 
                \tabitem gradient sparsity\\ 
            \end{tabular} & 
            \begin{tabular}{@{}l@{}}
                $\mathbb{E}\left[\left\|\rvg_t\right\|^2\right] = O\left(\sqrt{\frac{d}{T}}+\frac{d}{T}\right)$ \\
            \end{tabular} \\
        \cline{2-4}
            &
            \begin{tabular}{@{}l@{}}
                RMSProp,\\ Yogi~\cite{zaheer2018adaptive} 
            \end{tabular} & 
            \begin{tabular}{@{}l@{}}
                \tabitem bounded gradients\\ 
                \tabitem $\sigma^2$ bounded variance\\ 
            \end{tabular} & 
            \begin{tabular}{@{}l@{}}
                $\mathbb{E}\left[\left\|\rvg_t\right\|^2\right] = O\left(\frac{1}{T} + \sigma^2\right)$ \\
            \end{tabular} \\
        \cline{2-4}
            &
            \begin{tabular}{@{}l@{}}
                AdaGrad-\\ NORM~\cite{ward2018adagrad} 
            \end{tabular} & 
            \begin{tabular}{@{}l@{}}
                \tabitem bounded gradients\\
                \tabitem $\sigma^2$ bounded variance
            \end{tabular} & 
            \begin{tabular}{@{}l@{}}
                $\mathbb{E}\left[\left\|\rvg_t\right\|^2\right] = O\left(\frac{\ln T}{\sqrt{T}}\right)$
            \end{tabular} \\
        \cline{2-4}
            &
            \begin{tabular}{@{}l@{}}
                GGT~\cite{agarwal2019efficient} 
            \end{tabular} & 
            \begin{tabular}{@{}l@{}}
                \tabitem $\sigma^2$ bounded variance
            \end{tabular} & 
            \begin{tabular}{@{}l@{}}
                $\mathbb{E}\left[\left\|\rvg_t\right\|\right] = O\left(T^{-1/4}\right)$
            \end{tabular} \\
        \hline
            \multirow{2}{*}{
                \rotatebox{90}{\begin{tabular}{@{}c@{}}
                    Shuffling\\ Analysis
                \end{tabular}}
            } & 
            \begin{tabular}{@{}l@{}}
                Random\\ Shuffling\\ SGD~\cite{haochen2018random} 
            \end{tabular} & 
            \begin{tabular}{@{}l@{}}
                \tabitem strongly convex functions\\
                \tabitem bounded gradients\\ 
                \tabitem Hessian smoothness\\
            \end{tabular} & 
            \begin{tabular}{@{}l@{}}
                $\mathbb{E}\left[\left\|\vx_T - \vx^*\right\|^2\right] = O\left(\frac{1}{T^2} \right)$ 
            \end{tabular} \\
        \cline{2-4}
            &
            \begin{tabular}{@{}l@{}}
                Random\\ Shuffling\\ SGD~\cite{nguyen2020unified} 
            \end{tabular} & 
            \begin{tabular}{@{}l@{}}
                \tabitem bounded gradients\\
            \end{tabular} & 
            \begin{tabular}{@{}l@{}}
                $\mathbb{E}\left[\left\|\rvg_t\right\|^2\right] = O\left(\frac{\ln T}{T^{2/3}}\right)$ \\
            \end{tabular} \\
        \hline
            \multicolumn{2}{|l|}{
            \begin{tabular}{@{}c@{}}
                \radagrad
                (\textcolor{red}{ours})\\
                (Random SHuffling\\ AdaGrad)
                
            \end{tabular}} & 
            \begin{tabular}{@{}l@{}}
                \tabitem bounded gradients\\ 
                \tabitem $\sigma^2$ bounded variance\\
                \tabitem Outer product matrices\\ with full column rank and\\ bounded condition numbers
            \end{tabular} & 
            \begin{tabular}{@{}l@{}}
                $\color{red} \mathbb{E}\left[\left\|\rvg_t\right\|\right] = O\left(\frac{\sqrt{d}\ln T}{\sqrt{T}}\right)$
            \end{tabular} \\
        \hline
    \end{tabular}
    \vspace{-10pt}
\end{table}

\section{Notation and Preliminaries}
\label{sec:not}
In this section, we first introduce notation and preliminaries about objective functions, random shuffling and optimization methods including AdaGrad with full matrices~\cite{duchi2011adaptive} (AdaGrad\_F) and \radagrad.
Then, we list the commonly used assumptions required for the convergence rate analysis, and define the sufficient descent for the convenience of later explanation.

\textbf{Notation of objective functions.}
The objective function is defined in  eq.~\ref{eq:obj_intro},
where $n$ and $\grad f_i(\vx)$ denote the number of instances and the stochastic gradient for the $i$-th instance, respectively.
Besides, we call $\vx\in \mathbb{R}^d$ an $\epsilon$-approximate first-order stationary point, or simply an \emph{FSP}, if the gradient norm at $\vx$ satisfies $\left\|\grad f(\vx)\right\|\le \epsilon.$

\textbf{Notation of random shuffling.} 
Random shuffling is a sampling strategy to choose the stochastic mini-batch gradient at each iteration, which is different from the uniform sampling in vanilla SGD. 
Specifically, before the $t$-th epoch begins, random shuffling samples a random permutation $\sigma$ of the $n$ function uniformly and independently, and partitions $\sigma$ into several mini-batches $\left\{\sB^t_1, \sB^t_2, \ldots, \sB^t_m\right\}$ which satisfies $\sB^t_1\cup\sB^t_2\cup\ldots\cup\sB^t_m = \sI_n$ and $\sB^t_j\cap\sB^t_k =\emptyset, \forall j \not= k$. 
Then, the mini-batch gradient calculated at iteration $i$ in this epoch is denoted as
\begin{equation}
    \begin{aligned}[b]
        \grad f_{\sB^t_i}(\vx) \coloneqq \frac{1}{\left|\sB^t_i\right|}\sum_{k\in \sB^t_i}\grad f_k(\vx).
    \end{aligned}
\end{equation}
Without loss of generality, we set $|\sB_1^t|=|\sB_2^t| =\ldots = |\sB_m^t|$, i.e., all of the mini-batches have the same number of instances, in the following sections.

\textbf{Notation of optimization methods.}
We denote
$\vx_j^i$ as the parameter at the $j$-th iteration of the $i$-th epoch and 
\begin{equation}
    \label{eq:H_matrix_definition}
    \mH_{i,t} \coloneqq \left[\grad f_{\sB_1^t}(\vx_1^t), \grad f_{\sB_2^t}(\vx_2^t),\ldots, \grad f_{\sB_i^t}(\vx_i^t)\right]\in \sR^{d\times i},\quad i\le m,
\end{equation}
where $m$ and $d$ are the number of iterations in each epoch and the dimension of the parameters, respectively.
With the definition of $\mH_{i,t}$, we define the matrix
\begin{equation}
    \label{eq:G_matrix_definition}
    \mG_{i,t} \coloneqq \sum\limits_{\tau=1}^{t-1}\mH_{m,\tau}\mH_{m,\tau}^\top+\mH_{i,t}\mH_{i,t}^\top+\frac{\delta_{i,t}}{\Gamma}\mI,
\end{equation}  
where $\delta_{i,t}$ and $\Gamma$ are the perturbation and the scaling hyper-parameter to keep the positive-definite property for $\mG_{m,t}$. 
For any real matrix $\mM$, the maximum, the minimum and the $i$-th non-zero singular value are denoted as $\lambda_{\mathrm{min}}(\mM), \lambda_{\mathrm{max}}(\mM), \lambda_{i}(\mM)$, respectively.

Then, the iteration paradigm of both AdaGrad\_F and its variant \radagrad can be formulated as
\begin{equation}
    \label{eq:adagrad_update}
    \begin{split}
        \vx^t_{i+1} = \vx^t_i - \eta_{i,t}\mG^{-\frac{1}{2}}_{i,t}\grad f_{\sB_i^t}(\vx_i^t),
    \end{split}
\end{equation}
where $\frac{\delta_{i,t}}{\Gamma}$ in $\mG_{i,t}$ is a constant for AdaGrad\_F while adaptive for shuffled AdaGrad (\radagrad).

\textbf{Main assumptions.}
We list commonly used assumptions~\cite{fang2018spider, zaheer2018adaptive,ward2018adagrad} in the typical analysis as follows:
\begin{assumption}
    \label{ass:comm_ass_pre}
    We assume the following
    \begin{enumerate}
        \item The $\Delta\coloneqq f(\vx_1^1)-f^* <\infty$ where $f^*\coloneqq\inf_{\vx\in \sR^{d}}f(\vx)$ is the global infimum value of $f(\vx)$.
        \item \emph{(L-Smooth Assumption)} The component function $f_i(\vx)$ is $L$-smooth, i.e., for all $\vx,\vy\in \sR^{d}$ and $i\in \sI_n$, $\left\|\grad f_i(\vx)-\grad f_i(\vy)\right\|\le L\left\|\vx-\vy\right\|$.
        \item \emph{(Variance Bounded Assumption)} The stochastic gradient has a bounded variance, i.e., for any $i\in \sI_n$, $\mathbb{E}_i\left[\left\|\grad f_i(\vx) - \grad f(\vx)\right\|^2\right]\ge c_\sigma^2$.
        \item \emph{(Gradient Bounded Assumption)} The norm of stochastic gradient is upper bounded, i.e., for any $i\in \sI_n$, $\left\|\grad f_i(\vx)\right\|\le G$.
    \end{enumerate}
\end{assumption}

With L-Smooth Assumption in Assumption~\ref{ass:comm_ass_pre}, we next introduce the definition of sufficient descent, which plays an important role for understanding the core idea of this paper.
\begin{definition}
    \label{def:sufficient_descent}
    We denote the sufficient descent as the deterministic negative term in RHS of L-Smooth inequality about the objective function.
\end{definition}
For example, if we set the step size of \radagrad to be a fixed constant, we provide the sufficient descent about $\grad f(\vx_1^t)$ as follows.
\begin{equation}
    \label{ineq:convergence_bottleneck_1}
    \begin{aligned}[b]
        f(\vx_{m+1}^t)\mathop{\le}^{\circled{1}} & f(\vx_1^t) + \grad f^\top(\vx_1^t)\left(\vx_{m+1}^t - \vx_1^t\right) + \frac{L}{2}\left\|\vx_{m+1}-\vx_1^t\right\|^2\\
        \mathop{=}^{\circled{2}}& f(\vx_1^t) - \grad f^\top(\vx_1^t)\left[\sum\limits_{i=1}^m \eta\mG^{-\frac{1}{2}}_{i,t}\grad f_{\sB_i^t}(\vx_i^t) \right] + \frac{L}{2}\left\|\vx_{m+1}-\vx_1^t\right\|^2\\
        =&f(\vx_1^t) + \underbrace{\eta \grad f^\top(\vx_1^t)\left[\sum\limits_{i=1}^m\mG^{-\frac{1}{2}}_{m,t} \grad f_{\sB_i^t}(\vx_1^t)  -\sum\limits_{i=1}^m \mG^{-\frac{1}{2}}_{i,t}\grad f_{\sB_i^t}(\vx_i^t) \right]}_{\mathrm{a\ positive\ upper\ bound}} \\
        & +\underbrace{\frac{L}{2}\left\|\vx_{m+1}-\vx_1^t\right\|^2 }_{\mathrm{a\ positive\ upper\ bound}} - \underbrace{m\eta \grad f^\top(\vx_1^t)\mG_{m,t}^{-\frac{1}{2}}\grad f(\vx_1^t)}_{\mathrm{sufficient\ descent}},
    \end{aligned}
\end{equation}
where $\circled{1}$ follows from  L-Smooth Assumption in Assumption~\ref{ass:comm_ass_pre}, and $\circled{2}$ follows from eq.~\ref{eq:adagrad_update}.

\section{Core Idea: Reducible Gradient Perturbation Sequence}
\label{sec:rgps_ci}
In this section, we introduce the underlying ideas behind the convergence rate improvement for achieving first-order stationary points (FSPs) in non-convex optimization.
We introduce the concept of Reducible Gradient Perturbation Sequence (RGPS) which is defined as
\begin{equation}
    \label{eq:consturcted_auxiliary_sequence}
    \begin{split}
        \vs_t = \frac{1}{m}\sum\limits_{i=1}^m \grad f_{\sB_i^t}(\vx_i^t),\quad t\in \sI_T.
    \end{split}
\end{equation}
In random shuffling setting, RGPS can establish strong connections with both the update paradigm of adaptive gradient methods and full gradients, which provides the sufficient descent a $\Theta(\|\grad f(\vx_1^{\tau})\|)$ lower bound rather than the common result $\Theta(\|\grad f(\vx_1^{\tau})\|^2)$.
Besides, the upper bound $U(T)$ of sufficient descents in our analysis is similar to previous work, which means the sufficient conditions for convergence only request $U(T) \le \epsilon$ rather than $U(T) \le \epsilon^2$.
Hence, a better provable convergence rate can be obtained by introducing RGPS.
Specifically, we take AdaGrad\_F~\cite{duchi2011adaptive} and \radagrad as examples to explain how RGPS works in the convergence analysis, and organize the details guided by answering the following two questions
\begin{enumerate}
    \item How \radagrad obtains a tighter sufficient descent lower bound with RGPS?
    \item How the tighter sufficient descent lower bound benefits the convergence in \radagrad?
\end{enumerate}

\subsection{RGPS is a Coupling of Gradients and Sufficient Descents}
\label{sec:RGPS_sufficient_descent}
In this section, we answer the first question proposed in section~\ref{sec:rgps_ci}. 
We note that RGPS can establish strong connection with both the full gradient, i.e., $\grad f(\vx_1^t)$, and the sufficient descent about $\vs_t$.
Scaling the sufficient descent about $\vs_t$ through RPGS can utilize the properties of $\mG_{m,t}$ to the full potential, and obtain a tighter lower bound compared with investigating the sufficient descent about $\grad f(\vx_1^t)$ directly.

In particular, we provide two lemmas to explain that RGPS is a coupling of the full gradient sequence and the sufficient descent about $\vs_t$. 
\begin{lemma}
    \label{lem:auxiliary_sequence_closeness}
    Suppose Assumption~\ref{ass:comm_ass_pre} and Assumption~\ref{ass:H_full_column_rank_ass_pre} hold, we have
    \begin{equation}
        \label{ineq:auxiliary_sequence_closeness}
        \begin{split}
            \left\|\grad f(\vx_1^t) - \vs_t\right\| \le O(\eta /\sqrt{t}),
        \end{split}
    \end{equation}
    if the step size is fixed at each iteration.
\end{lemma}
This lemma illustrates that $\vs_t$ is close to $\grad f(\vx_1^t)$ when the fixed step size $\eta$ is small enough.
With triangle inequality, it also denotes that the full gradient norm $\|\grad f(\vx_1^t)\|$ can be bounded by $\|\vs_t\|$.
\begin{lemma}
    \label{lem:G_denominator_effect}
    Suppose Assumption~\ref{ass:comm_ass_pre} and Assumption~\ref{ass:H_full_column_rank_ass_pre} hold, in \radagrad, if $\delta_{m,t}\le t mG^2$, $0\le \delta_{m,t}-\delta_{m,t-1}\le m\lambda_{\mathrm{max}}\left(\mH_{m,t}\mH^\top_{m,t}\right)$ and $\Gamma \ge m$, we have
    \begin{equation}
        \label{eq:G_denominator_effect}
        \begin{split}
            \sqrt{\frac{\delta_{m,t}-\delta_{m,t-1}}{t}}\left\|\vs_t\right\| \le O\left(\vs_t^\top\mG^{-\frac{1}{2}}_{m,t}\vs_t\right).
        \end{split}
    \end{equation}
\end{lemma}
This lemma reveals the connection between $\|\vs_t\|$ and the sufficient descent about $\vs_\tau$.
According to the special structure of $\mG_{m,t}$, we are able to provide $\left\|\vs_t\right\|$ as the lower bound of the quadratic form $\vs_t^\top\mG_{m,t}^{-\frac{1}{2}}\vs_t$.
Combining Lemma~\ref{lem:auxiliary_sequence_closeness} with Lemma~\ref{lem:G_denominator_effect}, $\vs_t^\top\mG_{m,t}^{-\frac{1}{2}}\vs_t$ can be even lower bounded:
\begin{equation}
    \label{ineq:s_sufficient_descent_lb}
    \begin{aligned}[b]
        \sqrt{\frac{\delta_{m,t}- \delta_{m,t-1}}{t}}\left\|\grad f(\vx_1^t)\right\|= O\left(\vs_t^\top\mG_{m,t}^{-\frac{1}{2}}\vs_t\right) + O\left(\frac{\eta}{t}\right)
    \end{aligned}
\end{equation}
by using RPGS as a bridge.

On the other hand, if we investigate the sufficient descent about $\grad f(\vx_1^t)$ directly, we obtain a lower bound of the sufficient descent as
\begin{equation}
    \label{ineq:grad_sufficient_descent_lb}
    \begin{aligned}[b]
         \frac{1}{\sqrt{t}}\left\|\grad f(\vx_1^t)\right\|^2 \le O(\grad f^T(\vx_1^t)\mG_{m,t}^{-\frac{1}{2}}\grad f(\vx_1^t)),
    \end{aligned}
\end{equation}
due to the definition of $\mG_{m,t}$ and Gradient Bounded Assumption in Assumption~\ref{ass:comm_ass_pre}.
When the parameter $\vx_1^t$ is close to an FSP, $\|\grad f(\vx_1^t)\|$ is close to $0$ due to L-Smooth Assumption in Assumption~\ref{ass:comm_ass_pre}.
With a lower order of $\|\grad f(\vx_1^t)\|$, LHS of eq.~\ref{ineq:s_sufficient_descent_lb} is a undoubtedly better lower bound compared with that in eq.~\ref{ineq:grad_sufficient_descent_lb}, when $\delta_{m,t}-\delta_{m-1,t}$ has a constant lower bound, and the upper bound of RHS in eq.~\ref{ineq:s_sufficient_descent_lb} is almost the same as that in eq.~\ref{ineq:grad_sufficient_descent_lb}.

\subsection{Tight Lower Bounds Weaken Sufficient Conditions for the Convergence}
In this section, we answer the second question proposed in section~\ref{sec:rgps_ci}. 
First, we introduce the relation between the sufficient descent lower bound and the convergence rate.
Then, combining with section~\ref{sec:RGPS_sufficient_descent}, we provide a explanation about the convergence rate improvement in \radagrad.

\textbf{The relation between the sufficient descent lower bound and the convergence rate.} 
If we analyze the convergence through investigating the lower bound of the sufficient descent about $\grad f(\vx_1^t)$ in \radagrad, we have following inequalities:
\begin{equation}
    \label{ineq:convergence_bottleneck_3}
    \begin{split}
        m\eta\sum_{t=1}^T\frac{C}{\sqrt{t}}\left\|\grad f(\vx_1^t)\right\|^2 \mathop{\le}^{\circled{1}} m\eta\sum\limits_{t=1}^T\grad f^\top(\vx_1^t)\mG^{-\frac{1}{2}}_{m,t}\grad f(\vx_1^t) \mathop{\le}^{\circled{2}} f(\vx_1^1) - f^* + \Theta(T\eta^a),
    \end{split}
\end{equation}
where $\circled{1}$ follows from eq.~\ref{ineq:grad_sufficient_descent_lb} and $\circled{2}$ can be obtained through providing the telescoping sum of eq.~\ref{ineq:convergence_bottleneck_1} and scaling the terms with positive upper bounds to $\Theta(\eta^a)$ ($a\ge 0$).
If a random variable $\tau$ follows $\mathbb{P}[\tau = i] = \frac{i^{-0.5}}{\sum_{i=1}^T i^{-0.5}}$, we obtain 
\begin{small}
\begin{equation}
    \label{ineq:convergence_bottleneck_4}
    \begin{split}
        \mathbb{E}_{\tau} \left[\left\|\grad f(\vx_1^\tau)\right\|^2\right] \le \frac{f(\vx_1^1)-f^*}{\eta m \sqrt{T}}+\frac{\Theta(\sqrt{T}\eta^{a-1})}{m}
        \xRightarrow{\frac{1}{\eta\sqrt{T}}=\sqrt{T}\eta^{a-1}}
        \mathbb{E}_{\tau} \left[\left\|\grad f(\vx_1^\tau)\right\|^2\right] \le \Theta\left(T^{\frac{1}{a}-\frac{1}{2}}\right).
    \end{split}
\end{equation}
\end{small}
As a result, the sufficient condition for achieving \emph{FSP}, $\mathbb{E}_{\tau} [\|\grad f(\vx_1^\tau)\|^2] \le \epsilon^2$, is $RHS \le \epsilon^2$ for eq.~\ref{ineq:convergence_bottleneck_4}. 
The convergence rate of \radagrad is at least $O(T^{-1/4})$ if we lower bound the sufficient descent about $\grad f(\vx_1^\tau)$ like previous work.
As a result, we can conclude that the order of $\|\grad f(\vx_1^\tau)\|$ in the lower bound of sufficient descent directly decide the order of $\epsilon$ in RHS of the sufficient condition for convergence.
The order of $\|\grad f(\vx_1^\tau)\|$ higher, the convergence rate worse.

\textbf{The convergence rate improvement in \radagrad.} eq.~\ref{ineq:s_sufficient_descent_lb} in section~\ref{sec:RGPS_sufficient_descent} shows that the order of  $\|\grad f(\vx_1^t)\|$ in the lower bound of the sufficient descent about $\vs_t$ is significantly smaller than that in eq.~\ref{ineq:grad_sufficient_descent_lb}.
Hence, similar to eq.~\ref{ineq:convergence_bottleneck_3}, we can approximately provide
\begin{small}
\begin{equation}
    \label{ineq:convergence_bottleneck_5}
    \begin{aligned}[b]
         &\sum_{t=1}^T\frac{\eta}{\sqrt{t}}\left\|\grad f(\vx_1^t)\right\| \mathop{\le}^{\circled{1}} \eta \sum_{t=1}^T\vs_t^\top\mG_{m,t}^{-\frac{1}{2}}\vs_t + \Theta\left(\eta^2 \ln T\right) \mathop{\le}^{\circled{2}} f(\vx_1^1)-f^*+ \underbrace{\Theta\left(\eta^\alpha \sum_{t=1}^T t^{-\beta}\right)}_{T_1}  + \underbrace{\Theta\left(\eta^2 \ln T\right)}_{T_2},
    \end{aligned}
\end{equation}
\end{small}
where $\circled{1}$ follows from eq.~\ref{ineq:s_sufficient_descent_lb}, $\circled{2}$ can be obtained by techniques similar to the inequality $\circled{2}$ in eq.~\ref{ineq:convergence_bottleneck_3}.
The constants satisfy $\alpha\ge 0$, $\beta \ge 1$.
Notice that $T_1$ in eq.~\ref{ineq:convergence_bottleneck_5} is corresponding to $\Theta\left(T\eta^\alpha\right)$ in eq.~\ref{ineq:convergence_bottleneck_3}, and $T_2$ in eq.~\ref{ineq:convergence_bottleneck_5} is from the gap between $\vs_t$ and $\grad f(\vx_1^t)$. Similar to eq.~\ref{ineq:convergence_bottleneck_4}, we obtain
\begin{equation}
    \label{ineq:convergence_bottleneck_6}
    \begin{split}
        \mathbb{E}_{\tau}\left[\left\|\grad f(\vx_1^\tau)\right\|\right] \le \frac{f(\vx_1^1)-f^*}{\eta\sqrt{T}}+ \frac{\Theta(\eta^\alpha)}{\sqrt{T}} + \Theta\left(\frac{\eta\ln T}{\sqrt{T}}\right).
    \end{split}
\end{equation}
As a result, the sufficient condition for the parameters achieving \emph{FSP}, $\mathbb{E}_{\tau}\left[\left\|\grad f(\vx_1^\tau)\right\|\right]\le \epsilon$, is $RHS \le \epsilon$ for eq.~\ref{ineq:convergence_bottleneck_6}. 
That is to say, the convergence rate of \radagrad is near $\tilde{O}(T^{-1/2})$ which is better than previous best-known results.

\section{\radagrad achieves an \texorpdfstring{$\tilde{O}(T^{-1/2})$}{TEXT} Convergence Rate}
In this section, we show the convergence rate of adaptive gradient methods for achieving first-order stationary points (FSPs) in non-convex optimization can be $\tilde{O}(T^{-1/2})$. 
Note that our theoretical results are based on \radagrad, a variant of AdaGrad with full matrices (AdaGrad\_F), and is just proposed for analytic convenience.
Besides, we compare the total complexity between \radagrad and random shuffling SGD to illustrate that adaptive gradient methods can be faster than SGD after finite epochs, theoretically.
\begin{algorithm}[h]
    \SetAlgoLined
    \textbf{Input:} The step size $\eta>0$, the iteration number in one epoch $m$, the number of instances $n$\;
    \textbf{Variables:} $\mH_{i,t}\in \mathbb{R}^{d\times d}$, $\vg_i^t \in \mathbb{R}^{d\times d}$ $\delta_c = 0$\;
    \textbf{Initialization:} $\vx^0_{m+1}$, $\sigma_p = 0$\;
    \For{$t\gets1$ \KwTo $T$}{
        Initialize $\vx_1^t = \vx_{m+1}^{t-1}$, $\sigma_p = 0$\;
        \While{$\sigma_p < \frac{c_\sigma^2m^2}{8n}$}{
            Random shuffle the instances and get a partition $\left\{\sB_1^t, \sB_2^t, \ldots, \sB_m^t\right\}$\;
            Calculate $\sigma_p = \sum_{j=1}^m \left\|\grad f_{\sB_j^t}\left(\vx_1^t\right)\right\|^2$\;
        }    
        \For{$i\gets1$ \KwTo $m$}{
            Receive the mini-batch stochastic gradient $\vg_{i,t} = \grad f_{\sB^t_i}(\vx^t_i)$\;
            $\delta_c = \delta_c + \left\|\vg_{i,t}\right\|^2$\;
            Let $\mH_{i,t} = \left[\begin{matrix}\vg_{1,1} & \vg_{2,1} & \ldots & \vg_{i,t}
            \end{matrix}\right]$\;
            $\vx^t_{i+1} = \vx^t_i - \eta  \cdot \left(\mH_{i,t}\mH_{i,t}^\top+ \frac{\delta_c}{\Gamma} I\right)^{-\frac{1}{2}}\vg^t_i$\;
        }
    }
\caption{\radagrad with full matrices \label{alg:radagrad_full_matrices}}
\end{algorithm}

\textbf{\radagrad, a modified AdaGrad for theoretically analytic convenience.}
We list the main differences between \radagrad and AdaGrad as follows.
First, \radagrad requires a lower bound for the sum of mini-batch gradient norms, and obtains such lower bound with the sampling strategy (Step.~6 to Step.~8).
Second, AdaGrad\_F only considers the perturbation $\frac{\delta_c}{\Gamma}$ as a constant, while \radagrad has an adaptive perturbation which is related to the $l_2$ norm of mini-batch gradients (Step.~12).
On the other hand, Algorithm~\ref{alg:radagrad_full_matrices} which almost have a same update paradigm (Step.~14) as AdaGrad\_F.
Hence, it preserve benefits from second moments of adaptive gradient methods

In the following, we provide our additional mild assumptions, the convergence results and the total complexity of \radagrad.
Due to space limitations, the details of proof arguments are provided in the supplementary materials.
\begin{assumption}
    \label{ass:H_full_column_rank_ass_pre}
    We assume $d\gg m$, $\mH_{m,t}$ has full column rank and bounded condition number formulated as
    \begin{equation}
        \lambda_{\mathrm{max}}\left(\mH^\top_{m,t}\mH_{m,t}\right)/ \lambda_{\mathrm{min}}\left(\mH^\top_{m,t}\mH_{m,t}\right)\le c_{\kappa},
    \end{equation}
    where $d$ denotes the dimension of the parameters, and $m$ is the number of iterations in each epoch.
\end{assumption}
In Assumption~\ref{ass:H_full_column_rank_ass_pre}, $d\gg m$ follows the over-parameterized property in most neural network training.
Besides, we validate bounded condition numbers with experiments in section~\ref{sec:experiments}.
\begin{theorem}
    \label{thm:radagrad_convergence_rate}
    Under Assumption~\ref{ass:comm_ass_pre} and Assumption~\ref{ass:H_full_column_rank_ass_pre}, if $\eta\le \frac{c_\sigma^2}{16nLG}$, $\Gamma \ge m$ and the hyper-parameter $\delta_{j,i}$ satisfy
    \begin{equation}
        \label{eq:final_step_size_selection_ac}
        \begin{aligned}[b]
            \quad \delta_{j,i} = \sum_{p=1}^{i-1}\sum_{q=1}^m \left\|\grad f_{\sB_q^p}(\vx_q^p)\right\|^2 + \sum_{q=1}^j \left\|\grad f_{\sB_q^i}(\vx_q^i)\right\|^2, \forall\ i\in\sI_T, j\in\sI_m,
        \end{aligned}
    \end{equation}
    we have
    \begin{equation}
        \label{ineq:gradient_norm_convergence_4}
        \begin{split}
            \mathbb{E}_t\left[\left\|\grad f(\vx_1^t)\right\|\right] \le \frac{C_0}{\eta \sqrt{T}} + \frac{C_1}{\sqrt{T}} + \frac{C_2\eta}{\sqrt{T}} + \frac{C_3\eta^2}{\sqrt{T}} + \frac{C_4\ln(T)}{\sqrt{T}} + \frac{C_5\eta\ln(T)}{\sqrt{T}},
        \end{split}
    \end{equation}
    where $C_0, C_1, \ldots, C_5$  are constants which are independent with $T$ and defined in the proof.
\end{theorem}
Assuming that $L$, $G$ and $c_\sigma$ are known. 
Then, we can choose the following learning rate to obtain a concrete bound.
\begin{corollary}
    \label{cor:radagrad_complexity}
    Let $\left\{\vx_i^t\right\}$ be the sequence generated by Algorithm.~\ref{alg:radagrad_full_matrices} and $\vx_{out}$ be its output. 
    For given tolerance $\epsilon > 0$, under the same conditions as Theorem~\ref{thm:radagrad_convergence_rate}, if we choose $\eta = \frac{c_{\sigma}^2}{16nLG}$, $\Gamma=m$ and $m = n$, then to guarantee
    \begin{equation}
        \label{ineq:convergence_condition}
        \mathbb{E}_\tau\left[\left\|\grad f(\vx_1^\tau)\right\|\right] = \frac{\sum_{i=1}^T \frac{1}{\sqrt{i}}\left\|\sum_{j=1}^m \frac{1}{m} \grad f_{\sB_j^i}(\vx_1^i)\right\|}{\sum_{i=1}^T\frac{1}{\sqrt{i}}} \le \epsilon,
    \end{equation}
    it requires nearly $T = \lfloor 36C_{\mathrm{max}}n^3 d \epsilon^{-2}\rfloor$ outer iterations, where $C_{\mathrm{max}}$ is constant independent with $T$, $n$, $d$ and defined in the proof.
    In expectation, the total number of gradient evaluation is nearly
    \begin{equation*}
        \mathcal{T} = \Big\lfloor 36\left[1-\exp\left(-\frac{c_{\sigma}^{4}}{32G^4}\right)\right]^{-1}C_{\mathrm{max}}n^4d\epsilon^{-2} \Big\rfloor.
    \end{equation*}
\end{corollary}
To guarantee eq.~\ref{ineq:convergence_condition}, the total complexity required by random shuffling SGD is $O(C_{\mathrm{sgd}}n\epsilon^{-3})$. 
That is to say, for a rough comparison, if $\epsilon \le O(\frac{C_{\mathrm{sgd}}}{C_{\mathrm{max}}n^3d})$~\cite{haochen2018random,nguyen2020unified}, then Algorithm~\ref{alg:radagrad_full_matrices} seems to have advantages over random shuffling SGD in non-convex settings.
From this point of view, it seems that Algorithm~\ref{alg:radagrad_full_matrices} is inefficient when $n$ and $d$ is large.
However, our analysis focuses on explaining that the introduction of second moments is beneficial for adaptive gradient methods to reduce the dependence on $T$, and our convergence rate may be loose in that it does not take into account a tight dependence on $n$ and $d$ in our complexity results.
\vspace{-8pt}

\section{Experiments}
\label{sec:experiments}
In this section, we conduct comprehensive experiments to validate the additional mild assumption, i.e., Assumption~\ref{ass:H_full_column_rank_ass_pre}, and the acceleration effect from second moments.
 
The paper then proceeds to introduce the experimental settings for the image classification tasks. 
We used the CIFAR-10 dataset, and test a highly simplified CNN model, whose architecture can be found in our supplementary materials.
To compare convergence rates among SGD, AdaGrad, AdaGrad_F, \radagrad and their random shuffling version, e.g., SGD\_s, AdaGrad\_s, etc, we ran 200 epochs, and set the learning rate for different optimizers as theoretical suggested in Table~\ref{tab:twobest}.
\begin{table}[t]
    \caption{\small Hyper-parameters selection of different optimizers.}
        \label{tab:twobest}
        \centering
        \begin{tabular}{ | c | c | c |}
           \hline
           \textbf{Optimizers} & \textbf{Hyper-Parameters} & \textbf{Selection of $\eta$ } \\ \hline
           SGD_u & $\eta_t = \eta \cdot t^{-1/2}$ & $\left\{1.0, 0.1, 0.01\right\}$\\ \hline
           SGD_s & $\eta_t = \eta \cdot t^{-1/3}$ & $\left\{1.0, 0.1, 0.01\right\}$\\ \hline
           AdaGrad_u & $\eta_t = \eta \cdot t^{-1/2}$ & $\left\{0.1, 0.01, 0.001\right\}$\\
           \hline
           AdaGrad_s & $\eta_t = \eta \cdot t^{-1/2}$ & $\left\{0.1, 0.01, 0.001\right\}$\\
           \hline
           \radagrad_u & $\eta_t = \eta$, $\Gamma=d$ & $\left\{1.0, 0.1, 0.01\right\}$\\
           \hline
           \radagrad_s & $\eta_t = \eta$, $\Gamma=d$ & $\left\{1.0, 0.1, 0.01\right\}$\\
           \hline
           AdaGrad_F_u & $\eta_t = \eta$ & $\left\{1.0, 0.1, 0.01\right\}$\\
           \hline
           AdaGrad_F_s & $\eta_t = \eta$ & $\left\{1.0, 0.1, 0.01\right\}$\\
           \hline
        \end{tabular}
\end{table}

\begin{figure*}[!tb]
    \centering
       \includegraphics[width=0.5\textwidth]{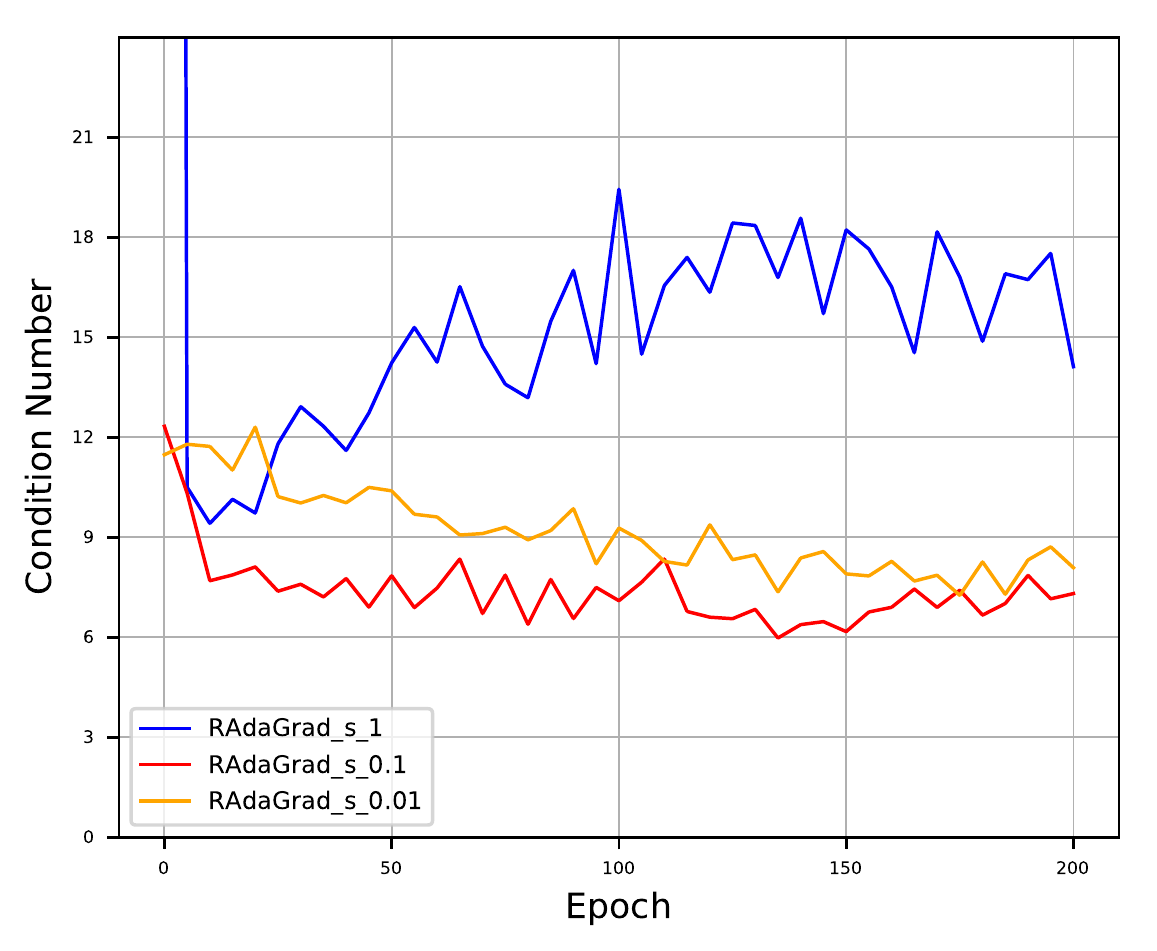}
        \caption{\small The condition number of \radagrad is bounded for different learning rates.}
        \label{fig:condtion_number}
\end{figure*}

\begin{figure*}[!tb]
    \centering
    \includegraphics[width=1.0\textwidth]{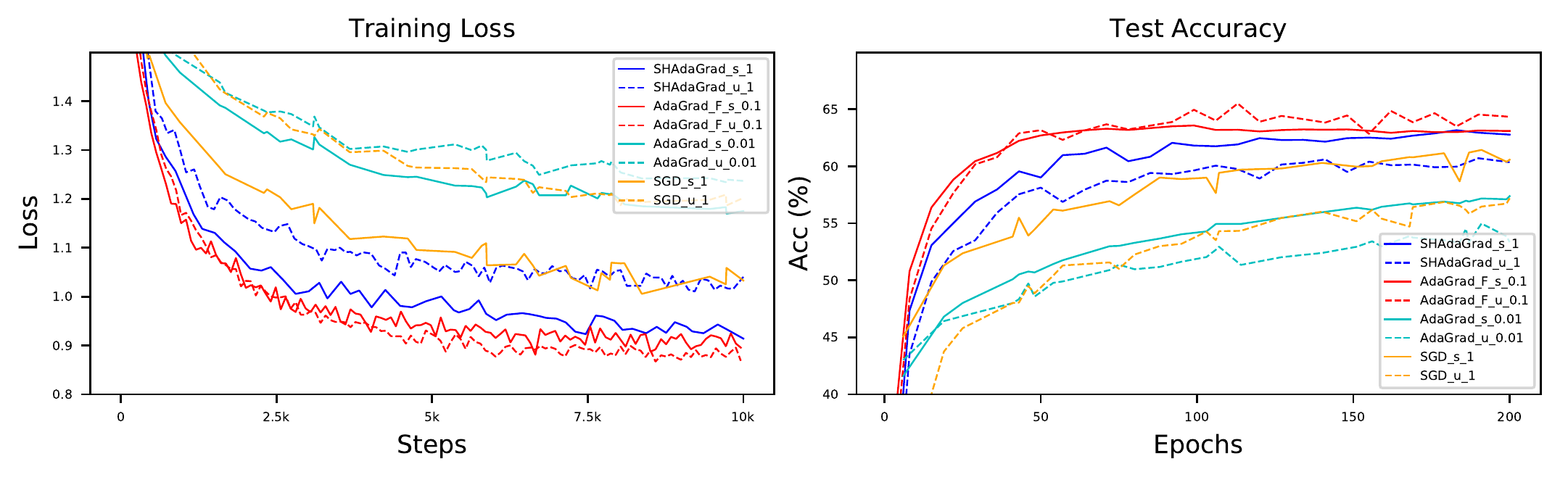}
    \caption{\small Convergence of optimizers (\_u) and their random shuffling version (\_s) on CIFAR-10 image classification tasks, which shows acceleration effect taken from the random shuffling and adaptive learning rates.}
    \label{fig:convergence_rate}
    \vspace{-15pt}
\end{figure*}
From fig.~\ref{fig:condtion_number}, we validate \emph{the condition number of $\mH_{i,t}$ will not increase with the number of iteration growth as Assumption~\ref{ass:H_full_column_rank_ass_pre} presented}.
From fig.~\ref{fig:convergence_rate}, we have two observations. First, \emph{random shuffling can actually take faster convergence for different optimizers in neural network training} except for AdaGrad\_F.
Second, \emph{adaptive gradient methods are usually faster than SGD in both uniform sampling and random shuffling settings}.
\vspace{-5pt}

\section{Conclusion}
\vspace{-6pt}
In this paper, we provide a novel perspective to illustrate that Adagrad variants can be faster than SGD after finite epochs in non-convex and random shuffling settings.
Under an additional mild assumption, we propose a minor revision of Adagrad, named \radagrad, and obtain a better convergence rate, i.e., $\tilde{O}(T^{-1/2})$, compared with previous best-known results.
Besides, we conduct extensive expeirments to validate the additional mild assumption and the acceleration effect taken from the introduction of second moments and random shuffling.  

\bibliographystyle{unsrt}  
\bibliography{references}  

\newpage
\appendix
\section{Notations and Assumptions for the Appendix}
In this section, we introduce some notations and assumptions used in this paper.
\subsection{Notations}
We denote the objective function as follows
\begin{equation}
    \label{def:objective_function_general_form}
    \min_{\vx\in \sR^d}\quad f(\vx)\coloneqq\mathbb{E}\left[F(\vx;\vzeta)\right],
\end{equation}
where $F(\vx,\vzeta)$ is the stochastic component indexed by some random variable $\vzeta$. 
$F(\vx,\vzeta)$ is smooth, and possibly non-convex. 
Let $\grad F(\vx,\vzeta)$ denote the stochastic gradient of $f(\vx)$.

The finite-sum objective is a special case of Eq.~\ref{def:objective_function_general_form} where $f(\vx)$ with finite sampled stochastic variables $\vzeta$.
It can be formulated as
\begin{equation}
    \min_{\vx \in \sR^d}\quad f(\vx) = \frac{1}{n}\sum_{i=1}^n f_i(\vx)
\end{equation}
whose stochastic gradient for the $i$-th instance is $\grad f_i(\vx)$.

Here, we describe the random shuffling setting in optimization procedure. 
Before each epoch, e.g., $i$-th epoch, beginning, we sample some permutation $\sigma_i$ of the set $\sI_n\coloneqq\left\{1, 2, \ldots, n\right\}$, and partitions $\sigma_i$ into mini-batch of equal size $\left\{\sB^i_1, \sB^i_2, \ldots, \sB^i_m\right\}$, where we require $\sB^i_1\cup\sB^i_2\cup\ldots\cup\sB^i_m = \sigma_i$ and $\sB^i_j\cap\sB^i_k =\emptyset, \forall j \not= k$. 
Then, the mini-batch gradient calculated at iteration $j$ in this epoch corresponds to $\grad f_{\sB^i_j}(\vx)$ denoted as
\begin{equation}
    \grad f_{\sB^i_j}(\vx) \coloneqq \frac{1}{\left|\sB^i_j\right|}\sum_{k\in \sB^i_j}f_k(\vx).
\end{equation}

Moreover, We denote $\vx_j^i$ as the parameter at the $j$-th iteration of the $i$-th epoch and 
\begin{equation}
    \label{eq:H_matrix_definition}
    \mH_{i,t} \coloneqq \left[\grad f_{\sB_1^t}(\vx_1^t), \grad f_{\sB_2^t}(\vx_2^t),\ldots, \grad f_{\sB_i^t}(\vx_i^t)\right]\in \sR^{d\times i},
\end{equation}
where $m$ and $d$ are presented as the number of iterations in each epoch and the dimension of the parameters, respectively.
With the definition of $\mH_{m,t}$, we define the matrix
\begin{equation}
    \label{eq:G_matrix_definition}
    \mG_{i,t} \coloneqq \sum\limits_{\tau=1}^{t-1}\mH_{m,\tau}\mH_{m,\tau}^\top+ \mH_{i,t}\mH_{i,t}^\top+\frac{\delta_{i,t}}{\Gamma}I,
\end{equation}  
where $\delta_{i,t}$ and $\Gamma$ are the perturbation and the scaling hyper-parameter to keep the positive-definite property for $\mG_{m,t}$.
Besides, for any real matrix $\mM$, we denote the maximum, the minimum and the $i$-th non-zero singular value as $\lambda_{max}(\mM), \lambda_{min}(\mM), \lambda_{i}(\mM)$, respectively.

\subsection{Assumptions}
In this subsection, we list our assumptions where Assumption~\ref{ass:comm_ass} introduces some common assumptions used in various previous work~\cite{fang2018spider, zaheer2018adaptive,ward2018adagrad}, and Assumption~\ref{ass:H_full_column_rank_ass} is required in our proof additionally.
To illustrate the rationality, we validate Assumption~\ref{ass:H_full_column_rank_ass} with various experiments.
\begin{assumption}
    \label{ass:comm_ass}
    We assume the following
    \begin{enumerate}
        \item The $\Delta\coloneqq f(\vx_1^1)-f^* <\infty$ where $f^*=\inf_{\vx\in \sR^{d}}f(\vx)$ is the global infimum value of $f(\vx)$.
        \item The component function $f_i(\vx)$ is $L$-smooth, i.e., for all $\vx,\vy\in \sR^{d}$ and $i\in \sI_n$, $\left\|\grad f_i(\vx)-\grad f_i(\vy)\right\|\le L\left\|\vx-\vy\right\|$.
        \item The stochastic gradient has a bounded variance, i.e., for any $i\in \sI_n$, $\mathbb{E}\left[\left\|\grad f_i(\vx) - \grad f(\vx)\right\|^2\right]\ge c_\sigma^2$.
        \item The norm of stochastic gradient is upper bounded, i.e., for any $i\in \sI_n$, $\left\|\grad f_i(\vx)\right\|\le \max\left\{c_{\sigma}, G^{\prime}\right\} \coloneqq G$.
    \end{enumerate}
\end{assumption}
\begin{assumption}
    \label{ass:H_full_column_rank_ass}
    Without loss of generality, we assume $d\gg m$, $\mH_{m,t}$ has full column rank and bounded condition number formulated as $$\frac{\lambda_{max}\left(\mH^\top_{m,t}\mH_{m,t}\right)}{\lambda_{min}\left(\mH^\top_{m,t}\mH_{m,t}\right)}\le c_{\kappa}$$.
\end{assumption}

\newpage
\section{Existing Lemmas}
\begin{lemma}[Conjugate Rule in~\cite{halko2011finding}]
    \label{lem:conjugate_rule}
    Suppose that $\mM \succeq 0$. For every $\mA$, the matrix $\mA^*\mM\mA \succeq 0$ where $\mA^*$ means the conjugate transpose matrix of $\mA$, In particular,
    \begin{equation}
        \label{eq:conjugate_rule}
        \begin{split}
            \mM\preceq \mN \Longrightarrow \mA^*\mM\mA\preceq \mA^*\mN\mA.
        \end{split}
    \end{equation}
\end{lemma}
\begin{lemma}[Hoeffding's inequality]
    \label{lem:hoeffding_s_inequality}
    Let $\rz_1,\rz_2,\ldots,\rz_n$ be independent bounded random variables with $\rz_i\in\left[a, b\right]$ for all $i$, where $-\infty<a<b<\infty$. Then
    \begin{equation}
        \label{ineq:hoeffding_s_inequality_1}
        \begin{split}
            \mathbb{P}\left[\frac{1}{n}\sum_{i=1}^n\left(\rz_i-\mathbb{E}\left[\rz_i\right]\right)\ge t\right]\le \exp\left(-\frac{2nt^2}{\left(b-a\right)^2}\right)
        \end{split}
    \end{equation}
    and 
    \begin{equation}
        \label{ineq:hoeffding_s_inequality_1}
        \begin{split}
            \mathbb{P}\left[\frac{1}{n}\sum_{i=1}^n\left(\rz_i-\mathbb{E}\left[\rz_i\right]\right)\le -t\right]\le \exp\left(-\frac{2nt^2}{\left(b-a\right)^2}\right)
        \end{split}
    \end{equation}
    for all $t\ge 0$.
\end{lemma}
\begin{lemma}[Lemma 13 in \cite{duchi2011adaptive}]
    \label{lem:matrix_root_positive_definite_lemma}
    Let $\mN\succeq \mM \succeq 0$ be symmetric $d\times d$ matrices. Then $\mN^{\frac{1}{2}}\succeq \mM^{\frac{1}{2}}$.
\end{lemma}
\begin{proof}
    This lemma had been proved in \cite{duchi2011adaptive}, we include a proof for the convenience of readers.
    Let $\lambda$ be a eigenvalue of $\mN^{\frac{1}{2}} - \mM^{\frac{1}{2}}$, corresponding to some eigenvector $\vx$. 
    Hence, we have $\left(\mN^{\frac{1}{2}} - \lambda I\right)\vx= \mM^{\frac{1}{2}}\vx $. Taking the inner product of both size with $\vx^\top\mN^{\frac{1}{2}}$, we have
    \begin{equation}
        \label{ineq:lem_B3_upper_bound}
        \begin{aligned}[b]
            &\underbrace{\vx^\top \mN \vx - \lambda \vx^\top \mN^{\frac{1}{2}}\vx}_{T_1} = \vx^\top \mN^{\frac{1}{2}}\left(\mN^{\frac{1}{2}} - \lambda I\right)\vx\\
            = & \vx^\top \mN^{\frac{1}{2}}\mM^{\frac{1}{2}}\vx  \le \left\|\mN^{\frac{1}{2}}\vx\right\|\left\|\mM^{\frac{1}{2}}\vx\right\| = \sqrt{\vx^\top \mN\vx \cdot \vx^\top \mM\vx} \le \underbrace{\vx^\top\mN \vx}_{T_2}.
        \end{aligned}
    \end{equation}
    Thus, with $T_1\le T_2$ and $\vx^\top \mN^{\frac{1}{2}}\vx\ge 0$, we obtain $\lambda\ge 0$ to complete the proof.
\end{proof}
\begin{lemma}
    \label{lem:matrix_inverse_positive_definite_lemma}
    Let $\mN\succeq \mM \succeq 0$ be symmetric $d\times d$ matrices. Then $\mN^{-1}\preceq \mM^{-1}$.
\end{lemma}
\begin{proof}
    Since $\mN \succeq \mM$, we have $\mM^{-\frac{1}{2}}\mN \mM^{-\frac{1}{2}} = \left(\mM^{-\frac{1}{2}}\mN^{\frac{1}{2}}\right)\left(\mN^{\frac{1}{2}}\mM^{-\frac{1}{2}}\right)\succeq \mI$ because of Lemma~\ref{lem:conjugate_rule}. 
    Commuting the product of two matrice does not change the eigenvalues, hence all eigenvalues of $\mN^{\frac{1}{2}}\mM^{-1}\mN^{\frac{1}{2}}$ are larger than $1$. 
    Utilizing Lemma~\ref{lem:conjugate_rule} again, then we obtain $\mM^{-1}\succeq \mN^{-1}$ to complete the proof.
\end{proof}
\begin{lemma}[Sherman-Morrison formula]
    \label{lem:sherman_morrison_formula}
    Suppose $\mM\in\mathbb{R}^{d\times d}$ is an invertible square matrix and $\vu, \vv\in \mathbb{R}^{d}$ are column vectors. Then $\mM+\vu\vv^\top$ is invertible if and only if $1+\vv^\top\mM\vu \not= 0$. In this case,
    \begin{equation}
        \label{eq:sherman_morrison_formula}
        \left(\mM + \vu\vv^\top\right)^{-1} = \mM^{-1} - \frac{\mM^{-1}\vu\vv^\top\mM^{-1}}{1+\vv^\top\mM^{-1}\vu}
    \end{equation}
\end{lemma}
\newpage
\section{Important Lemmas}
\begin{lemma}
    \label{lem:positive_definite_property_in_one_epoch}
    In \radagrad, suppose that the Assumption~\ref{ass:comm_ass} hold, and the perturbation satisfies $0\le \delta_{i+1,t} - \delta_{i,t} \le G^2$. For any $1\le i \le j \le m$ and $\Gamma\ge 1$, we have
    $$\mG_{j,t}^{\frac{1}{2}} \preceq \mG_{i,t}^{\frac{1}{2}} + \sqrt{2m}G\cdot{\mI}$$.
\end{lemma}
\begin{proof}
    It can be easily checked when $i=j$ holds. 
    Therefore, we only need to prove $i<j$.
    To simplify notations in the following proof, we set $\Delta\mG \coloneqq \sum_{k=i+1}^j \grad f_{\sB_k^t}(\vx_k^t)\grad f^\top_{\sB_k^t}(\vx_k^t)$.
    Then, according to the definition, we have
    \begin{equation}
        \label{eq:matrix_G_update_iteration}
        \begin{aligned}[b]
            \mG_{j,t} = \mG_{i,t} + \Delta\mG + \frac{\delta_{j,t}-\delta_{i,t}}{\Gamma}\cdot \mI. 
        \end{aligned}
    \end{equation}
    With the following fact, 
    \begin{equation}
        \label{ineq:delta_G_l2_norm_inequlities}
        \begin{aligned}[b]
            \lambda_{max}\left(\Delta\mG\right) \le tr\left(\Delta\mG\right) \mathop{\le}^{\circled{1}} (j-i)G^2 \le mG^2\quad \mathrm{and}\quad \delta_{j,t}-\delta_{i,t} = \sum_{k=i}^{j-1}\left(\delta_{k+1,t} - \delta_{k,t}\right) \le (j-i)G^2\le mG^2,
        \end{aligned}
    \end{equation}
    where $\circled{1}$ follows from the  gradient bounded condition, the forth item in Assumption~\ref{ass:comm_ass}, we have
    \begin{equation}
        \label{ineq:delat_G_positive_definite_inequality}
        \begin{aligned}
        & \Delta\mG + \frac{\delta_{j,t}-\delta_{i,t}}{\Gamma}\cdot I \preceq m\left(1+\frac{1}{\Gamma}\right)G^2 \cdot \mI \mathop{\preceq}^{\circled{1}} 2mG^2 \cdot\mI,
        \end{aligned}
    \end{equation}
    where $\circled{1}$ follows from the fact $\Gamma \ge 1$.
    Then, we obtain
    \begin{equation}
        \begin{aligned}[b]
            & \mG_{j,t} \mathop{\preceq}^{\circled{1}} \mG_{i,t}+2mG^2\cdot \mI\\
            \mathop{\Rightarrow}^{\circled{2}}\  &\mG_{j,t}^{\frac{1}{2}} \preceq \left(\mG_{i,t} + 2mG^2 \cdot \mI\right)^{\frac{1}{2}} \mathop{\preceq}^{\circled{3}} \mG_{i,t}^{\frac{1}{2}} + \sqrt{2m}G\cdot{\mI},
        \end{aligned}
    \end{equation}
    where $\circled{1}$ follows from Eq.~\ref{eq:matrix_G_update_iteration} and Eq.~\ref{ineq:delat_G_positive_definite_inequality}, $\circled{2}$ follows from Lemma~\ref{lem:matrix_root_positive_definite_lemma} and $\circled{3}$ follows from the following fact
    \begin{equation}
        \begin{aligned}[b]
            \left[\left(\mG_{i,t} + 2mG^2 \cdot \mI\right)^{\frac{1}{2}} \right]^2 = \mG_{i,t} + 2mG^2 \cdot \mI \preceq \mG_{i,t} + 2\sqrt{2m}G\mG_{i,t}^{\frac{1}{2}} + 2mG^2 \cdot \mI = \left(\mG_{i,t}^{\frac{1}{2}} + \sqrt{2m}G\cdot{\mI}\right)^2 ,
        \end{aligned}
    \end{equation}
    and Lemma~\ref{lem:matrix_root_positive_definite_lemma}.
    Thus, we complete the proof.
\end{proof}
\begin{lemma}
    \label{lem:bound_of_neighbor_parameters}
    In \radagrad, suppose that the Assumption~\ref{ass:comm_ass} hold, and the perturbation $\delta_{i,t}$ is no decreasing. For any $1\le i \le m$, we have
    $$\left\|\vx_{i+1}^t - \vx_i^t\right\|^2 \le \min\left\{\eta_{i,t}^2, \eta_{i,t}^2G^2\lambda^{-1}_{min}\left(\mG_{1,t}\right)\right\}$$.
\end{lemma}
\begin{proof}
    According to iterations in \radagrad, if we set $\overline{\mG}_{i,t} \coloneqq \mG_{i,t} - \grad f_{\sB_i^t}(\vx_i^t)\grad f^\top_{\sB_i^t}(\vx_i^t)$, we have
    \begin{equation}
        \label{ineq:bound_of_neighbor_parameters}
        \begin{aligned}[b]
            & \left\|\vx^t_{i+1} - \vx^t_i\right\|^2 = \left\|\eta_{i,t}\mG_{i,t}^{-\frac{1}{2}}\grad f_{\sB_i^t}(\vx_i^t)\right\|^2 = \underbrace{\eta_{i,t}^2 \grad f^\top_{\sB_i^t}(\vx_i^t)\mG_{i,t}^{-1}\grad f_{\sB_i^t}(\vx_i^t)}_{T_1}\\
            \le & \eta_{i,t}^2 \grad f^\top_{\sB_i^t}(\vx_i^t)\left(\overline{\mG}_{i,t} + \grad f_{\sB_i^t}(\vx_i^t)\grad f^\top_{\sB_i^t}(\vx_i^t)\right)^{-1}\grad f_{\sB_i^t}(\vx_i^t)\\
            \mathop{\le}^{\circled{1}} &\eta_{i,t}^2 \left(\grad f^\top_{\sB_i^\top}(\vx_i^t)\overline{\mG}_{i,t}^{-1}\grad f_{\sB_i^t}(\vx_i^t) -\frac{\left\|\grad f^\top_{\sB_i^t}(\vx_i^t)\overline{\mG}_{i,t}^{-1}\grad f_{\sB_i^t}(\vx_i^t)\right\|^2}{1 + \grad f^\top_{\sB_i^t}(\vx_i^t)\overline{\mG}_{i,t}^{-1}\grad f_{\sB_i^t}(\vx_i^t)} \right) =  \frac{\eta_{i,t}^2 \grad f^\top_{\sB_i^t}(\vx_i^t)\overline{\mG}_{i,t}^{-1}\grad f_{\sB_i^t}(\vx_i^t)}{1 + \grad f^\top_{\sB_i^t}(\vx_i^t)\overline{\mG}_{i,t}^{-1}\grad f_{\sB_i^t}(\vx_i^t)}\le \eta_{i,t}^2,
        \end{aligned}
    \end{equation}
    where $\circled{1}$ follows from Lemma~\ref{lem:sherman_morrison_formula}.
    For $T_1$ in Eq.~\ref{ineq:bound_of_neighbor_parameters}, we also have
    \begin{equation}
        \label{ineq:bound_of_neighbor_parameters_T1}
        \begin{aligned}[b]
            T_1 \mathop{\le}^{\circled{1}}  \eta_{i,t}^2 \grad f^\top_{\sB_i^t}(\vx_i^t)\mG_{1,t}^{-1}\grad f_{\sB_i^t}(\vx_i^t) \le \frac{\eta_{i,t}^2}{\lambda_{min}\left(\mG_{1,t}\right)}\left\|\grad f_{\sB_i^t}(\vx_i^t)\right\|^2 \mathop{\le}^{\circled{2}}  \eta_{i,t}^2G^2\lambda^{-1}_{min}\left(\mG_{1,t}\right),
        \end{aligned}
    \end{equation}
    where $\circled{1}$ follows from the fact $\mG_{i,t}\succeq \mG_{1,t}$ and Lemma~\ref{lem:matrix_inverse_positive_definite_lemma}, $\circled{2}$ follows from the the gradient bounded condition, i.e., the forth item in Assumption~\ref{ass:comm_ass}.
    Combining Eq.~\ref{ineq:bound_of_neighbor_parameters} with Eq.~\ref{ineq:bound_of_neighbor_parameters_T1}, we obtain
    \begin{equation}
        \label{ineq:bound_of_neighbor_parameters_final}
        \begin{aligned}[b]
            \left\|\vx_{i+1}^t - \vx_i^t\right\|^2 \le \min\left\{\eta_{i,t}^2, \eta_{i,t}^2G^2\lambda^{-1}_{min}\left(\mG_{1,t}\right)\right\}
        \end{aligned}
    \end{equation}
    to complete the proof.
\end{proof}
\begin{corollary}
    \label{lem:bound_of_parameters}
    In \radagrad, suppose that the Assumption~\ref{ass:comm_ass} hold, and the perturbation $\delta_{i,t}$ is no decreasing. For any $1\le i \le m$, we have
    \begin{equation}
        \label{ineq:bound_of_parameters}
        \left\|\vx_i^t - \vx_1^t\right\|^2\le \min\left\{\eta_{.,t}^2\lambda_{min}^{-1}\left(\mG_{1,t}\right)(i-1)^2G^2, \eta_{.,t}^2(i-1)^2\right\},
    \end{equation}
    when step size satisfies $\eta_{1,t} = \eta_{2,t} = \ldots = \eta_{m,t} = \eta_{.,t}$.
\end{corollary}
\begin{proof}
    It can be easily checked that $i=1$ holds in Eq.~\ref{ineq:bound_of_parameters}.
    Therefore, we only need to prove $i\ge 2$.
    We have
    \begin{equation}
        \label{ineq:upper_bound_diff_parameters}
        \begin{split}
            \left\|\vx_i^t - \vx_1^t\right\|^2=\left\|\sum_{j=1}^{i-1}\left(\vx_{j+1}^t - \vx_j^t\right)\right\|^2 \mathop{\le}^{\circled{1}} \left[\sum_{j=1}^{i-1}\left\|\vx_{j+1}^t - \vx_j^t\right\|\right]^2 \mathop{\le}^{\circled{2}} (i-1)\left[\sum_{j=1}^{i-1}\left\|\vx_{j+1}^t - \vx_j^t\right\|^2\right],
        \end{split}
    \end{equation}
    where $\circled{1}$ follows from the triangle inequality and $\circled{2}$ follows from the Cauchy-Schwarz inequality. 
    According to the iteration of \radagrad, for any $j\in \sI_{i-1}$ we have
    \begin{equation}
        \label{ineq:upper_bound_close_parameters}
        \begin{aligned}[b]
            &\left\|\vx_{j+1}^t - \vx_j^t\right\|^2 \mathop{\le}^{\circled{1}} \min\left\{\eta_{.,t}^2, \eta_{.,t}^2G^2\lambda^{-1}_{min}\left(\mG_{1,t}\right)\right\},
        \end{aligned}
    \end{equation}
    where $\circled{1}$ follows from Lemma~\ref{lem:bound_of_neighbor_parameters}.
    Combining Eq.~\ref{ineq:upper_bound_diff_parameters} with Eq.~\ref{ineq:upper_bound_close_parameters}, we obtain
    \begin{equation}
        \label{ineq:bound_of_parameters_final}
        \begin{split}
            \left\|\vx_i^t - \vx_1^t\right\|^2\le \min\left\{\eta_{.,t}^2\lambda_{min}^{-1}\left(\mG_{1,t}\right)(i-1)^2G^2, \eta_{.,t}^2(i-1)^2\right\}.
        \end{split}
    \end{equation}
    Thus, we complete the proof.
\end{proof}
\begin{lemma}
    \label{lem:G_trace_lower_bound}
    In \radagrad, suppose that the Assumption~\ref{ass:comm_ass} hold, and the perturbation $\delta_{i,t}$ is no decreasing. If the step size in the $i$-th epoch satisfies $\eta_{.,i}\le \frac{c_\sigma^2}{16nLG}$, we have $$\sum_{j=1}^m \left\|\grad f_{\sB_j^i}(\vx_j^i)\right\|^2\ge \frac{c_\sigma^2m^2}{16n}$$ with probability at least $1-\exp\left(-\frac{m^2c^4_{\sigma}}{32n^2G^4}\right)$.
\end{lemma}
\begin{proof}
    To simplify notations in the following proof, we set $s_i=\sum_{j=1}^m \left\|\grad f_{\sB_j^i}(\vx_1^i)\right\|^2$.
    We have
    \begin{equation}
        \label{ineq:si_lower_bound}
        \begin{split}
            \mathbb{E}_{\sigma_i}\left[s_i\right]=\sum_{j=1}^m\left(\mathbb{E}_{\sB_{j}^i}\left[\left\|\grad f_{\sB_j^i}(\vx_1^i)\right\|^2\right]\right).
        \end{split}
    \end{equation}
    With the symmetry of the permutation $\sigma_i$, there is $\mathbb{P}\left[\sB_1^i = \tilde{\sigma}\right]=\mathbb{P}\left[\sB_2^i = \tilde{\sigma}\right]=\ldots =\mathbb{P}\left[\sB_m^i = \tilde{\sigma}\right]$ for any specific subset $\tilde{\sigma} \subseteq \sI_m$ where $\left|\tilde{\sigma}\right| = \left|\sB_j^i\right|, \forall j\in\sI_m$.
    Thus, when the sample size of mini-batch $\left|\tilde{\sigma}\right|=n/m\le (n+1)/2$, the expectation $\mathbb{E}_{\sigma_i}\left[s_i\right]$ can be reformulated as
    \begin{equation}
        \label{ineq:exp_si_lower_bound}
        \begin{aligned}[b]
            &\mathbb{E}_{\sigma_i}\left[s_i\right]=m\mathbb{E}_{\tilde{\sigma}}\left[\left\|\grad f_{\tilde{\sigma}}(\vx_1^i)\right\|^2\right] =m\mathbb{E}_{\tilde{\sigma}}\left[\left\|\grad f_{\tilde{\sigma}}(\vx_1^i)-\mathbb{E}_{\tilde{\sigma}}\left[\grad f_{\tilde{\sigma}}(\vx_1^i)\right]+\mathbb{E}_{\tilde{\sigma}}\left[\grad f_{\tilde{\sigma}}(\vx_1^i)\right]\right\|^2\right]\\
            = & m\left[\mathbb{E}_{\tilde{\sigma}}\left[\left\|\grad f_{\tilde{\sigma}}(\vx_1^i)-\mathbb{E}_{\tilde{\sigma}}\left[\grad f_{\tilde{\sigma}}(\vx_1^i)\right]\right\|^2\right] + \mathbb{E}_{\tilde{\sigma}}\left[\left\|\mathbb{E}_{\tilde{\sigma}}\left[\grad f_{\tilde{\sigma}}(\vx_1^i)\right]\right\|^2\right]\right]\\
            \ge &m\mathbb{E}_{\tilde{\sigma}}\left[\left\|\grad f_{\tilde{\sigma}}(\vx_1^i)-\mathbb{E}_{\tilde{\sigma}}\left[\grad f_{\tilde{\sigma}}(\vx_1^i)\right]\right\|^2\right] \mathop{=}^{\circled{1}} m\cdot \frac{c^2_{\sigma}m}{n} \cdot \left(1-\frac{n/m-1}{n-1}\right) \ge \frac{c^2_{\sigma}m^2}{2n},
        \end{aligned}
    \end{equation}
    where $\circled{1}$ is established because of the property of sampling without replacement variance. 
    Besides, for any $i\in\sI_m$, we have $s_i\in \left[0,mG^2\right]$.
    According to Lemma~\ref{lem:hoeffding_s_inequality}, we have
    \begin{equation}
        \label{ineq:s_i_hoeffding_bound}
        \begin{split}
            \mathbb{P}\left[\rs_i-\frac{c^2_{\sigma}m^2}{4n}\le -\frac{c^2_{\sigma}m^2}{8n}\right] \le \mathbb{P}\left[\rs_i-\mathbb{E}\left[\rs_i\right]\le -\frac{c^2_{\sigma}m^2}{8n}\right]\le \exp\left(-\frac{m^2c^4_{\sigma}}{32n^2G^4}\right).
        \end{split}
    \end{equation}
    That is to say, $s_i\ge \frac{c^2_{\sigma}m^2}{8n}$ establishes with probability at least $1-\exp\left(-\frac{m^2c^4_{\sigma}}{32n^2G^4}\right)$. 
    When the step size in the $i$-th epoch is small enough, i.e., $\eta_{.,i}\le \frac{c_\sigma^2}{16nLG}$, we have
    \begin{equation}
        \label{ineq:minor_term_upper_bound}
        \begin{aligned}[b]
            &\sum_{j=1}^m \left\|\grad f_{\sB_j^i}(\vx_1^i)\right\|\cdot \left\|\grad f_{\sB_j^i}(\vx_j^i) - \grad f_{\sB_j^i}(\vx_1^i)\right\|\\
            \le & G\sum_{j=1}^m \left\|\grad f_{\sB_j^i}(\vx_j^i) - \grad f_{\sB_j^i}(\vx_1^i)\right\|\le LG\sum_{j=1}^m\left\|\vx_j^i - \vx_1^i\right\|\\
            \le & LG\sum_{j=1}^m \sum_{k=1}^{j-1}\left\|\vx_{k+1}^i - \vx_k^i\right\| \mathop{\le}^{\circled{1}} LG \sum_{j=1}^m \sum_{k=1}^{j-1}\eta_{.,i}=LG\eta_{.,i} \sum_{j=1}^m (j-1)\\
            \le & \frac{LG\eta_{.,i}m^2}{2} \le \frac{c_\sigma^2 m^2}{32n},
        \end{aligned}
    \end{equation}
    where $\circled{1}$ follows from Lemma~\ref{lem:bound_of_neighbor_parameters}.
    Thus, we have
    \begin{equation}
        \label{ineq:trace_lower_bound}
        \begin{aligned}[b]
            &\sum_{j=1}^m \left\|\grad f_{\sB_j^i}(\vx_j^i)-\grad f_{\sB_j^i}(\vx_1^i) + \grad f_{\sB_j^i}(\vx_1^i)\right\|^2 \ge \sum_{j=1}^m \left[\left\|\grad f_{\sB_j^i}(\vx_1^i)\right\| - \left\|\grad f_{\sB_j^i}(\vx_j^i)-\grad f_{\sB_j^i}(\vx_1^i)\right\|\right]^2\\
            \ge & \sum_{j=1}^m \left\|\grad f_{\sB_j^i}(\vx_1^i)\right\|^2 - 2\sum_{j=1}^m \left\|\grad f_{\sB_j^i}(\vx_1^j)\right\|\cdot \left\|\grad f_{\sB_j^i}(\vx_j^i) -\grad f_{\sB_j^i}(\vx_1^i)\right\| \mathop{\ge}^{\circled{1}} \frac{c_\sigma^2m^2}{16n}
        \end{aligned}
    \end{equation}
    with probability at least $1-\exp\left(-\frac{m^2c^4_{\sigma}}{32n^2G^4}\right)$, where we have $\circled{1}$ due to Eq.~\ref{ineq:s_i_hoeffding_bound} and Eq.~\ref{ineq:minor_term_upper_bound}. Then, the proof is completed.
\end{proof}

\newpage
\section{Convergence Rate of \radagrad on Non-Convex and Shuffling Settings}
\begin{lemma}
    \label{lem:sufficient_descent}
    In \radagrad, suppose that the Assumption~\ref{ass:comm_ass} and Assumption~\ref{ass:H_full_column_rank_ass} hold, If the hyper-parameter $\delta_{j,i}$, $\Gamma$ and the step size $\eta$ satisfy
    \begin{equation}
        \label{eq:final_step_size_selection_ac_0}
        \begin{aligned}[b]
            \delta_{j,i} = \sum_{p=1}^{i-1}\sum_{q=1}^m \left\|\grad f_{\sB_q^p}(\vx_q^p)\right\|^2 + \sum_{q=1}^j \left\|\grad f_{\sB_q^i}(\vx_q^i)\right\|^2, \forall\ i\in\sI_T, j\in\sI_m,\quad 1\le \Gamma\le n \quad \mathrm{and}\quad \eta \le \frac{c_\sigma^2}{16nLG}.
        \end{aligned}
    \end{equation}
    Then, we have
    \begin{equation}
        \label{eq:sufficient_descent}
        \begin{aligned}[b]
            &\frac{\eta_{.,t}}{4m} \left(\sum_{i=1}^m \grad f_{\sB_i^t}(\vx_i^t)\right)^\top\mG_{m,t}^{-\frac{1}{2}}\left(\sum_{i=1}^m \grad f_{\sB_i^t}(\vx_i^t)\right)\\
            \le &f(\vx_1^t) - f(\vx_{m+1}^t) + \min\left\{\frac{2\eta_{.,t}^3L^2G^2m^3}{3\lambda_{min}^{1.5}\left(\mG_{1,t}\right)}, \frac{2\eta^3_{.,t}L^2m^3}{3\sqrt{\lambda_{min}\left(\mG_{m,t}\right)}}\right\}+ \min\left\{\frac{3\eta_{.,t}m^{1.5}G^3}{\sqrt{2}\lambda_{min}\left(\mG_{1,t}\right)}, \frac{3\eta_{.,t}m^{1.5}G}{\sqrt{2}}\right\},
        \end{aligned}
    \end{equation}   
    for the $t-$th epoch in \radagrad.
\end{lemma}
\begin{proof}
    According to the iteration of \radagrad, when $1\le i \le j \le m$, the outer product matrix $\mG$s have the following properties
    \begin{equation}
        \label{ineq:matrix_G_properties}
        \begin{aligned}[b]
            \mG_{j,t}\succeq \mG_{i,t} \mathop{\Rightarrow}^{\circled{1}} \mG_{j,t}^{\frac{1}{2}} \succeq \mG_{i,t}^{\frac{1}{2}} \mathop{\Rightarrow}^{\circled{2}} \mG_{i,t}^{-\frac{1}{2}} \succeq \mG_{m,t}^{-\frac{1}{2}} ,
        \end{aligned}
    \end{equation}
    where $\circled{1}$ follows from Lemma~\ref{lem:matrix_root_positive_definite_lemma}, $\circled{2}$ follows from Lemma~\ref{lem:matrix_inverse_positive_definite_lemma}.

    With the $L-$Lipschitz continuous gradient assumption, the second item in Assumption~\ref{ass:comm_ass}, we have
    \begin{equation}
        \label{ineq:l_lipschitz_in_loss_function_1}
        \begin{aligned}[b]
            f(\vx_{m+1}^t) - f(\vx_1^t) \le& \grad f^\top(\vx_1^t)\left(\vx_{m+1}^t - \vx_{1}^t\right) + \frac{L}{2}\left\|\vx_{m+1}^t-\vx_1^t \right\|^2\\
            =&\grad f^\top(\vx_1^t)\left[\sum_{i=1}^m -\eta_{i,t}\mG_{i,t}^{-\frac{1}{2}}\grad f_{\sB_i^t}(\vx_i^t)\right] + \frac{L}{2}\left\|\vx_{m+1}^t-\vx_1^t \right\|^2\\
            =&\grad f^\top(\vx_1^t)\left[\sum_{i=1}^m \left(-\eta_{i,t}\mG_{i,t}^{-\frac{1}{2}}\grad f_{\sB_i^t}(\vx_i^t) + \eta_{i,t}\mG_{m,t}^{-\frac{1}{2}}\grad f_{\sB_i^t}(\vx_i^t)\right)\right] + \frac{L}{2}\left\|\vx_{m+1}^t-\vx_1^t \right\|^2\\
            &-\sum_{i=1}^m \eta_{i,t}\grad f^\top(\vx_1^t)\mG_{m,t}^{-\frac{1}{2}}\grad f_{\sB_i^t}(\vx_i^t).
        \end{aligned}
    \end{equation}
    We set the step sizes of different iterations to be the same in one epoch, i.e., $\eta_{1,t} = \eta_{2,t} = \ldots = \eta_{m,t}  = \eta_{.,t}$. Then, we obtain
    \begin{equation}
        \label{ineq:l_lipschitz_in_loss_function_2}
        \begin{aligned}[b]
            f(\vx_{m+1}^t) - f(\vx_1^t) \le& \eta_{.,t}\grad f^\top(\vx_1^t)\left[\sum_{i=1}^m\left(\mG_{m,t}^{-\frac{1}{2}}\grad f_{\sB_i^t}(\vx_i^t)-\mG_{i,t}^{-\frac{1}{2}}\grad f_{\sB_i^t}(\vx_i^t)\right)\right]+ \frac{L}{2}\left\|\vx_{m+1}^t-\vx_1^t \right\|^2\\
            & -\eta_{.,t}\left(\grad f(\vx_1^t) - \frac{1}{m}\sum_{i=1}^m\grad f_{\sB_i^t}(\vx_i^t)+\frac{1}{m}\sum_{i=1}^m\grad f_{\sB_i^t}(\vx_i^t)\right)^\top\mG_{m,t}^{-\frac{1}{2}}\left(\sum_{i=1}^m \grad f_{\sB_i^t}(\vx_i^t)\right)\\
            =&\underbrace{\eta_{.,t}\grad f^\top(\vx_1^t)\left[\sum_{i=1}^m\left(\mG_{m,t}^{-\frac{1}{2}}-\mG_{i,t}^{-\frac{1}{2}}\right)\grad f_{\sB_i^t}(\vx_i^t)\right]}_{T_1}+ \frac{L}{2}\left\|\vx_{m+1}^t-\vx_1^t \right\|^2\\
            &+\underbrace{\frac{\eta_{.,t}}{m}\left[\sum_{i=1}^m\left(\grad f_{\sB_i^t}(\vx_i^t)-\grad f_{\sB_i^t}(\vx_1^t)\right)\right]^\top\mG_{m,t}^{-\frac{1}{2}}\left(\sum_{i=1}^m \grad f_{\sB_i^t}(\vx_i^t)\right)}_{T_2}\\
            &-\frac{\eta_{.,t}}{m}\left(\sum_{i=1}^m \grad f_{\sB_i^t}(\vx_i^t)\right)^\top\mG_{m,t}^{-\frac{1}{2}}\left(\sum_{i=1}^m \grad f_{\sB_i^t}(\vx_i^t)\right)
        \end{aligned}
    \end{equation}
    We next bound $T_1$ and $T_2$ separately. 
    First for $T_1$ in Eq.~\ref{ineq:l_lipschitz_in_loss_function_2}, we have
    \begin{equation}
        \label{ineq:T1_of_LLF_2}
        \begin{aligned}[b]
        T_1 = & \eta_{.,t}\left(\grad f(\vx_1^t) - \frac{1}{m}\sum_{i=1}^m\grad f_{\sB_i^t}(\vx_i^t)+ \frac{1}{m}\sum_{i=1}^m\grad f_{\sB_i^t}(\vx_i^t)\right)^\top\left[\sum_{i=1}^m\left(\mG_{m,t}^{-\frac{1}{2}}-\mG_{i,t}^{-\frac{1}{2}}\right)\grad f_{\sB_i^t}(\vx_i^t)\right]\\
        =& \underbrace{\eta_{.,t}\left(\grad f(\vx_1^t) - \frac{1}{m}\sum_{i=1}^m\grad f_{\sB_i^t}(\vx_i^t)\right)^\top\left[\sum_{i=1}^m\left(\mG_{m,t}^{-\frac{1}{2}}-\mG_{i,t}^{-\frac{1}{2}}\right)\grad f_{\sB_i^t}(\vx_i^t)\right]}_{T_{1.1}}\\
        &+\underbrace{\frac{\eta_{.,t}}{m}\left(\sum_{i=1}^m\grad f_{\sB_i^t}(\vx_i^t)\right)^\top\left[\sum_{i=1}^m\left(\mG_{m,t}^{-\frac{1}{2}}-\mG_{i,t}^{-\frac{1}{2}}\right)\grad f_{\sB_i^t}(\vx_i^t)\right]}_{T_{1.2}}.
        \end{aligned}
    \end{equation}
    For $T_{1.1}$ we have
    \begin{equation}
        \label{ineq:T11_of_LLF_2}
        \begin{aligned}[b]
            T_{1.1}\mathop{=}^{\circled{1}}&\eta_{.,t}\left(\grad f(\vx_1^t) - \frac{1}{m}\sum_{i=1}^m\grad f_{\sB_i^t}(\vx_i^t)\right)^\top\mG_{m,t}^{-\frac{1}{4}}\mG_{m,t}^{\frac{1}{4}}\left[\sum_{i=1}^m\left(\mG_{m,t}^{-\frac{1}{2}}-\mG_{i,t}^{-\frac{1}{2}}\right)\grad f_{\sB_i^t}(\vx_i^t)\right]\\
            \mathop{\le}^{\circled{2}} & \frac{\eta_{.,t}}{m}\left[\sum_{i=1}^m\left(\grad f_{\sB_i^t}(\vx_i^t)-\grad f_{\sB_i^t}(\vx_1^t)\right)\right]^\top\mG_{m,t}^{-\frac{1}{2}}\left[\sum_{i=1}^m\left(\grad f_{\sB_i^t}(\vx_i^t)-\grad f_{\sB_i^t}(\vx_1^t)\right)\right]\\
            &+ \frac{\eta_{.,t}}{4m} \left[\sum_{i=1}^m\left(\mG_{m,t}^{-\frac{1}{2}}-\mG_{i,t}^{-\frac{1}{2}}\right)\grad f_{\sB_i^t}(\vx_i^t)\right]^\top\mG_{m,t}^{\frac{1}{2}}\left[\sum_{i=1}^m\left(\mG_{m,t}^{-\frac{1}{2}}-\mG_{i,t}^{-\frac{1}{2}}\right)\grad f_{\sB_i^t}(\vx_i^t)\right],
        \end{aligned}
    \end{equation}
    where $\circled{1}$ follows from the invertibility of $\mG_{m,t}$ and $\circled{2}$ follows from the Cauchy-Schwarz inequality.
    
    Similarly, for $T_{1.2}$ we have
    \begin{equation}
        \label{ineq:T12_of_LLF_2}
        \begin{aligned}[b]
            T_{1.2}=& \frac{\eta_{.,t}}{m}\left(\sum_{i=1}^t\grad f_{\sB_i^t}(\vx_i^t)\right)^\top\mG_{m,t}^{-\frac{1}{4}}\mG_{m,t}^{\frac{1}{4}}\left[\sum_{i=1}^m\left(\mG_{m,t}^{-\frac{1}{2}}-\mG_{i,t}^{-\frac{1}{2}}\right)\grad f_{\sB_i^t}(\vx_i^t)\right]\\
            \le & \eta_{.,t}\left[\frac{1}{4m}\left(\sum_{i=1}^m\grad f_{\sB_i^t}(\vx_i^t)\right)^\top\mG_{m,t}^{-\frac{1}{2}}\left(\sum_{i=1}^m\grad f_{\sB_i^t}(\vx_i^t)\right)\right.\\
            &\left.+\frac{1}{m}\left[\sum_{i=1}^m\left(\mG_{m,t}^{-\frac{1}{2}}-\mG_{i,t}^{-\frac{1}{2}}\right)\grad f_{\sB_i^t}(\vx_i^t)\right]^\top \mG_{m,t}^{\frac{1}{2}}\left[\sum_{i=1}^m\left(\mG_{m,t}^{-\frac{1}{2}}-\mG_{i,t}^{-\frac{1}{2}}\right)\grad f_{\sB_i^t}(\vx_i^t)\right]\right].
        \end{aligned}
    \end{equation}
    Submitting Eq.~\ref{ineq:T11_of_LLF_2} and Eq.~\ref{ineq:T12_of_LLF_2} backinto Eq.~\ref{ineq:T1_of_LLF_2}, we obtain
    \begin{equation}
        \label{ineq:T1_of_LLF_2_final}
        \begin{aligned}[b]
            T_1 \le &\frac{\eta_{.,t}}{m}\left[\sum_{i=1}^m\left(\grad f_{\sB_i^t}(\vx_i^t)-\grad f_{\sB_i^t}(\vx_1^t)\right)\right]^\top\mG_{m,t}^{-\frac{1}{2}}\left[\sum_{i=1}^m\left(\grad f_{\sB_i^t}(\vx_i^t)-\grad f_{\sB_i^t}(\vx_1^t)\right)\right]\\
            & + \frac{5\eta_{.,t}}{4m}\left[\sum_{i=1}^m\left(\mG_{m,t}^{-\frac{1}{2}}-\mG_{i,t}^{-\frac{1}{2}}\right)\grad f_{\sB_i^t}(\vx_i^t)\right]^\top \mG_{m,t}^{\frac{1}{2}}\left[\sum_{i=1}^m\left(\mG_{m,t}^{-\frac{1}{2}}-\mG_{i,t}^{-\frac{1}{2}}\right)\grad f_{\sB_i^t}(\vx_i^t)\right]\\
            & + \frac{\eta_{.,t}}{4m}\left(\sum_{i=1}^m\grad f_{\sB_i^t}(\vx_i^t)\right)^\top\mG_{m,t}^{-\frac{1}{2}}\left(\sum_{i=1}^m\grad f_{\sB_i^t}(\vx_i^t)\right).
        \end{aligned}
    \end{equation}

    Similarly, $T_2$ in Eq.~\ref{ineq:l_lipschitz_in_loss_function_2} satisfies
    \begin{equation}
        \label{ineq:T2_of_LLF_2}
        \begin{aligned}[b]
            T_2 = &\frac{\eta_{.,t}}{m}\left[\sum_{i=1}^m\left(\grad f_{\sB_i^t}(\vx_i^t)-\grad f_{\sB_i^t}(\vx_1^t)\right)\right]^\top\mG_{m,t}^{-\frac{1}{4}}\mG_{m,t}^{-\frac{1}{4}}\left(\sum_{i=1}^m \grad f_{\sB_i^t}(\vx_i^t)\right)\\
            \le &\frac{\eta_{.,t}}{m}\left[\sum_{i=1}^m\left(\grad f_{\sB_i^t}(\vx_i^t)-\grad f_{\sB_i^t}(\vx_1^t)\right)\right]^\top\mG_{m,t}^{-\frac{1}{2}}\left[\sum_{i=1}^m\left(\grad f_{\sB_i^t}(\vx_i^t)-\grad f_{\sB_i^t}(\vx_1^t)\right)\right]\\
            &+\frac{\eta_{.,t}}{4m} \left(\sum_{i=1}^m \grad f_{\sB_i^t}(\vx_i^t)\right)^\top\mG_{m,t}^{-\frac{1}{2}}\left(\sum_{i=1}^m \grad f_{\sB_i^t}(\vx_i^t)\right).
        \end{aligned}
    \end{equation}
    
    As a result, we plug Eq.~\ref{ineq:T1_of_LLF_2_final} and Eq.~\ref{ineq:T2_of_LLF_2} into Eq.~\ref{ineq:l_lipschitz_in_loss_function_2}, and obtain that
    \begin{equation}
        \label{ineq:sufficient_descent_start}
        \begin{aligned}[b]
            f(\vx_{m+1}^t) - f(\vx_1^t)\le &\underbrace{\frac{2\eta_{.,t}}{m}\left[\sum_{i=1}^m\left(\grad f_{\sB_i^t}(\vx_i^t)-\grad f_{\sB_i^t}(\vx_1^t)\right)\right]^\top\mG_{m,t}^{-\frac{1}{2}}\left[\sum_{i=1}^m\left(\grad f_{\sB_i^t}(\vx_i^t)-\grad f_{\sB_i^t}(\vx_1^t)\right)\right]}_{T_1}\\
            &+ \underbrace{\frac{5\eta_{.,t}}{4m}\left[\sum_{i=1}^m\left(\mG_{m,t}^{-\frac{1}{2}}-\mG_{i,t}^{-\frac{1}{2}}\right)\grad f_{\sB_i^t}(\vx_i^t)\right]^\top \mG_{m,t}^{\frac{1}{2}}\left[\sum_{i=1}^m\left(\mG_{m,t}^{-\frac{1}{2}}-\mG_{i,t}^{-\frac{1}{2}}\right)\grad f_{\sB_i^t}(\vx_i^t)\right]}_{T_2}\\
            &-\frac{\eta_{.,t}}{2m} \left(\sum_{i=1}^m \grad f_{\sB_i^t}(\vx_i^t)\right)^\top\mG_{m,t}^{-\frac{1}{2}}\left(\sum_{i=1}^m \grad f_{\sB_i^t}(\vx_i^t)\right)+ \underbrace{\frac{L}{2}\left\|\vx_{m+1}^t-\vx_1^t \right\|^2}_{T_3}.
        \end{aligned}
    \end{equation}

    For $T_1$ in Eq.~\ref{ineq:sufficient_descent_start}, we have
    \begin{equation}
        \label{ineq:T1_for_SD}
        \begin{aligned}[b]
            T_1 =&\frac{2\eta_{.,t}}{m}\left\|\sum_{i=1}^m \mG_{m,t}^{-\frac{1}{4}}\left(\grad f_{\sB_i^t}(\vx_i^t) - \grad f_{\sB_i^t}(\vx_1^t)\right)\right\|^2\\
            \le &2\eta_{.,t}\sum_{i=1}^m\left[\left(\grad f_{\sB_i^t}(\vx_i^t) - \grad f_{\sB_i^t}(\vx_1^t)\right)^T\mG_{m,t}^{-\frac{1}{2}}\left(\grad f_{\sB_i^t}(\vx_i^t) - \grad f_{\sB_i^t}(\vx_1^t)\right)\right]\\
            \le & \frac{2\eta_{.,t}}{\sqrt{\lambda_{min}\left(\mG_{m,t}\right)}}\sum_{i=1}^m\left\|\grad f_{\sB_i^t}(\vx_i^t) - \grad f_{\sB_i^t}(\vx_1^t)\right\|^2\\
            \mathop{\le}^{\circled{1}} & \frac{2\eta_{.,t}L^2}{\sqrt{\lambda_{min}\left(\mG_{m,t}\right)}}\sum_{i=1}^m\left\|\vx_i^t - \vx_1^t\right\|^2 \mathop{\le}^{\circled{2}} \frac{2\eta_{.,t}L^2}{\sqrt{\lambda_{min}\left(\mG_{m,t}\right)}}\sum_{i=1}^m \left(\min\left\{\lambda_{min}^{-1}\left(\mG_{1,t}\right)G^2, 1\right\}\cdot \eta_{.,t}^2(i-1)^2\right) \\
            = & \frac{2\eta^3_{.,t}L^2}{\sqrt{\lambda_{min}\left(\mG_{m,t}\right)}}\cdot \min\left\{\lambda_{min}^{-1}\left(\mG_{1,t}\right)G^2, 1\right\} \cdot \sum_{i=1}^m\left(i-1\right)^2\le \min\left\{\frac{2\eta_{.,t}^3L^2G^2m^3}{3\lambda_{min}^{1.5}\left(\mG_{1,t}\right)}, \frac{2\eta^3_{.,t}L^2m^3}{3\sqrt{\lambda_{min}\left(\mG_{m,t}\right)}}\right\}\\ 
        \end{aligned}
    \end{equation}
    where $\circled{1}$ follows from the $L$-Lipschitz continuous gradient assumption, the second item in Assumption~\ref{ass:comm_ass}, and $\circled{2}$ follows from Corollary.~\ref{lem:bound_of_parameters}.

    With similar techniques, we relax $T_2$ in Eq.~\ref{ineq:sufficient_descent_start} as follows
    \begin{equation}
        \label{ineq:T2_for_SD}
        \begin{aligned}[b]
            T_2 = &\frac{5\eta_{.,t}}{4m}\left\|\sum_{i=1}^m\mG_{m,t}^{\frac{1}{4}}\left(\mG_{m,t}^{-\frac{1}{2}}-\mG_{i,t}^{-\frac{1}{2}}\right)\grad f_{\sB_i^t}(\vx_i^t)\right\|^2\\
            \mathop{\le}^{\circled{1}} & \frac{5\eta_{.,t}}{4}\sum_{i=1}^m\left\|\mG_{m,t}^{\frac{1}{4}}\left(\mG_{m,t}^{-\frac{1}{2}}-\mG_{i,t}^{-\frac{1}{2}}\right)\grad f_{\sB_i^t}(\vx_i^t)\right\|^2\\
            = & \frac{5\eta_{.,t}}{4}\sum_{i=1}^m  \grad f^\top_{\sB_i^t}(\vx_i^t) \left(\mG_{m,t}^{-\frac{1}{2}}-\mG_{i,t}^{-\frac{1}{2}}\right)\mG_{m,t}^{\frac{1}{2}}\left(\mG_{m,t}^{-\frac{1}{2}}-\mG_{i,t}^{-\frac{1}{2}}\right)\grad f_{\sB_i^t}(\vx_i^t),
        \end{aligned}
    \end{equation}
    where $\circled{1}$ follows from the Cauchy-Schwarz inequality.
    For each $i$ in the last equation of Eq.~\ref{ineq:T2_for_SD}, we have
    \begin{equation}
        \label{ineq:T2_for_SD_component1}
        \begin{aligned}[b]
            &\grad f^\top_{\sB_i^t}(\vx_i^t) \left(\mG_{m,t}^{-\frac{1}{2}}-\mG_{i,t}^{-\frac{1}{2}}\right)\mG_{m,t}^{\frac{1}{2}}\left(\mG_{m,t}^{-\frac{1}{2}}-\mG_{i,t}^{-\frac{1}{2}}\right)\grad f_{\sB_i^t}(\vx_i^t)\\
            = & \grad f^\top_{\sB_i^t}(\vx_i^t)\left(\mG_{m,t}^{-\frac{1}{2}}-2\mG_{i,t}^{-\frac{1}{2}} + \mG_{i,t}^{-\frac{1}{2}}\mG_{m,t}^{\frac{1}{2}}\mG_{i,t}^{-\frac{1}{2}}\right)\grad f_{\sB_i^t}(\vx_i^t)\\
            \mathop{\le}^{\circled{1}} & \grad f^\top_{\sB_i^t}(\vx_i^t)\left(\mG_{i,t}^{-\frac{1}{2}}\mG_{m,t}^{\frac{1}{2}}\mG_{i,t}^{-\frac{1}{2}}-\mG_{i,t}^{-\frac{1}{2}}\right)\grad f_{\sB_i^t}(\vx_i^t) = \grad f^\top_{\sB_i^t}(\vx_i^t)\mG_{i,t}^{-\frac{1}{2}}\left(\mG_{m,t}^{\frac{1}{2}}-\mG_{i,t}^{\frac{1}{2}}\right)\mG_{i,t}^{-\frac{1}{2}}\grad f_{\sB_i^t}(\vx_i^t)\\
            \mathop{\le}^{\circled{2}} &\sqrt{2m}G \grad f^\top_{\sB_i^t}(\vx_i^t)\mG_{i,t}^{-1}\grad f_{\sB_i^t}(\vx_i^t) \mathop{\le}^{\circled{3}} \min\left\{\frac{\sqrt{2m}G^3}{\lambda_{min}\left(\mG_{i,t}\right)}, \sqrt{2m}G\right\} \le \min\left\{\frac{\sqrt{2m}G^3}{\lambda_{min}\left(\mG_{1,t}\right)},\sqrt{2m}G\right\}
        \end{aligned}
    \end{equation}
    where $\circled{1}$ follows from $\mG^{-\frac{1}{2}}_{m,t} - \mG^{-\frac{1}{2}}_{i,t}\preceq 0$ stated in Eq.~\ref{ineq:matrix_G_properties}, $\circled{2}$ follows from Lemma~\ref{lem:positive_definite_property_in_one_epoch} and $\circled{3}$ follows from the gradient upper bound assumption, the forth point in Assumption~\ref{ass:comm_ass} and Lemma~\ref{lem:sherman_morrison_formula}. 
    After submitting Eq.~\ref{ineq:T2_for_SD_component1} back into Eq.~\ref{ineq:T2_for_SD}, we have
    \begin{equation}
        \label{ineq:T2_for_SD_final}
        \begin{split}
            T_2\le \min\left\{\frac{5\eta_{.,t}m^{1.5}G^3}{2\sqrt{2}\lambda_{min}\left(\mG_{1,t}\right)}, \frac{5\eta_{.,t}m^{1.5}G}{2\sqrt{2}}\right\}.
        \end{split}
    \end{equation}

    For $T_3$ in Eq.~\ref{ineq:sufficient_descent_start}, we have
    \begin{equation}
        \label{ineq:T3_for_SD}
        \begin{aligned}[b]
            T_3 = & \frac{L}{2}\left\|\sum_{i=1}^m -\eta_{.,t}\mG^{-\frac{1}{2}}_{i,t}\grad f_{\sB_i^t}(\vx_i^t)\right\|^2 = \frac{L\eta_{.,t}^2}{2}\left\|\sum_{i=1}^m \left(\mG^{-\frac{1}{2}}_{i,t}-\mG_{m,t}^{-\frac{1}{2}}\right)\grad f_{\sB_i^t}(\vx_i^t) + \mG_{m,t}^{-\frac{1}{2}}\sum_{i=1}^m \grad f_{\sB_i^t}(\vx_i^t)\right\|^2\\
            \le & L\eta_{.,t}^2\left[\sum_{i=1}^m \grad f_{\sB_i^t}(\vx_i^t)\right]^\top\mG^{-1}_{m,t}\left[\sum_{i=1}^m \grad f_{\sB_i^t}(\vx_i^t)\right] + L\eta_{.,t}^2 \left\|\sum_{i=1}^m \left(\mG^{-\frac{1}{2}}_{i,t}-\mG_{m,t}^{-\frac{1}{2}}\right)\grad f_{\sB_i^t}(\vx_i^t)\right\|^2.
        \end{aligned}
    \end{equation}
    Hence, if the step size is small enough, we then obtain
    \begin{equation}
        \label{ineq:T3_for_SD_component1}
        \begin{aligned}[b]
            &\eta_{.,t}\le \frac{c^2_{\sigma}}{16nLG} \mathop{\le}^{\circled{1}} \frac{c_{\sigma}}{16nL}\le \frac{c_{\sigma}}{16\sqrt{n\Gamma}L} = \frac{\frac{c_{\sigma}m}{4\sqrt{n\Gamma}}}{4mL} \mathop{\le}^{\circled{2}} \frac{\sqrt{\frac{\delta_{m,t}-\delta_{m,t-1}}{\Gamma}}}{4mL}  \le \frac{\lambda_{min}\left(\mG_{m,t}^{\frac{1}{2}}\right)}{4mL} \\
            \Rightarrow &\eta_{.,t}\mI \preceq \frac{\mG^{\frac{1}{2}}_{m,t}}{4mL} \mathop{\Rightarrow}^{\circled{3}} L\eta_{.,t}^2\mG^{-1}_{m,t} \preceq \frac{\eta_{.,t}}{4m}\mG_{m,t}^{-\frac{1}{2}}\\
            \Rightarrow & \left(\sum_{i=1}^m \grad f_{\sB_i^t}(\vx_i^t)\right)^\top L\eta_{.,t}^2\mG^{-1}_{m,t}\left(\sum_{i=1}^m \grad f_{\sB_i^t}(\vx_i^t)\right) \le \left(\sum_{i=1}^m\grad f_{\sB_i^t}(\vx_i^t)\right)^\top \frac{\eta_{.,t}}{4m}\mG_{m,t}^{-\frac{1}{2}}\left(\sum_{i=1}^m\grad f_{\sB_i^t}(\vx_i^t)\right),
        \end{aligned}
    \end{equation}
    where $\circled{1}$ follows from the definition of constant $G$, $\circled{2}$ follows from Lemma~\ref{lem:G_trace_lower_bound}, and
    $\circled{3}$ follows from Lemma~\ref{lem:conjugate_rule}.
    Besides, with the same step size upper bound, we have
    \begin{equation}
        \label{ineq:T3_for_SD_component2}
        \begin{aligned}[b]
            &  L\eta_{.,t}^2 \left\|\sum_{i=1}^m \left(\mG^{-\frac{1}{2}}_{i,t}-\mG_{m,t}^{-\frac{1}{2}}\right)\grad f_{\sB_i^t}(\vx_i^t)\right\|^2 \le \frac{\eta_{.,t}}{4m} \left\|\sum_{i=1}^m\mG_{m,t}^{\frac{1}{4}}\left(\mG_{m,t}^{-\frac{1}{2}}-\mG_{i,t}^{-\frac{1}{2}}\right)\grad f_{\sB_i^t}(\vx_i^t)\right\|^2\\
            &\mathop{\le}^{\circled{1}} \frac{\eta_{.,t}}{4}\sum_{i=1}^m  \grad f^\top_{\sB_i^t}(\vx_i^t) \left(\mG_{m,t}^{-\frac{1}{2}}-\mG_{i,t}^{-\frac{1}{2}}\right)\mG_{m,t}^{\frac{1}{2}}\left(\mG_{m,t}^{-\frac{1}{2}}-\mG_{i,t}^{-\frac{1}{2}}\right)\grad f_{\sB_i^t}(\vx_i^t)\mathop{\le}^{\circled{2}} \min\left\{\frac{\eta_{.,t}m^{1.5}G^3}{2\sqrt{2}\lambda_{min}\left(\mG_{1,t}\right)}, \frac{\eta_{.,t}m^{1.5}G}{2\sqrt{2}}\right\},
        \end{aligned}
    \end{equation}
    where $\circled{1}$ follows from Eq.~\ref{ineq:T2_for_SD} and $\circled{2}$ follows from similar techniques with Eq.~\ref{ineq:T2_for_SD_component1}.
    Hence, combining Eq.~\ref{ineq:T3_for_SD_component1}, Eq.~\ref{ineq:T3_for_SD_component2} with Eq.~\ref{ineq:T3_for_SD}, we obtain
    \begin{equation}
        \label{ineq:T3_for_SD_final}
        \begin{split}
            T_3\le \frac{\eta_{.,t}}{4m}\left(\sum_{i=1}^m\grad f_{\sB_i^t}(\vx_i^t)\right)^\top\mG_{m,t}^{-\frac{1}{2}}\left(\sum_{i=1}^m\grad f_{\sB_i^t}(\vx_i^t)\right) + \min\left\{\frac{\eta_{.,t}m^{1.5}G^3}{2\sqrt{2}\lambda_{min}\left(\mG_{1,t}\right)}, \frac{\eta_{.,t}m^{1.5}G}{2\sqrt{2}}\right\}.
        \end{split}
    \end{equation}

    Submitting Eq.~\ref{ineq:T1_for_SD}, Eq.~\ref{ineq:T2_for_SD_final} and Eq.~\ref{ineq:T3_for_SD_final} back into Eq.~\ref{ineq:sufficient_descent_start}, we obtain
    \begin{equation}
        \label{ineq:sufficient_descent_end}
        \begin{aligned}[b]
            &\frac{\eta_{.,t}}{4m} \left(\sum_{i=1}^m \grad f_{\sB_i^t}(\vx_i^t)\right)^\top\mG_{m,t}^{-\frac{1}{2}}\left(\sum_{i=1}^m \grad f_{\sB_i^t}(\vx_i^t)\right)\\
            \le &f(\vx_1^t) - f(\vx_{m+1}^t) + \min\left\{\frac{2\eta_{.,t}^3L^2G^2m^3}{3\lambda_{min}^{1.5}\left(\mG_{1,t}\right)}, \frac{2\eta^3_{.,t}L^2m^3}{3\sqrt{\lambda_{min}\left(\mG_{m,t}\right)}}\right\}+ \min\left\{\frac{3\eta_{.,t}m^{1.5}G^3}{\sqrt{2}\lambda_{min}\left(\mG_{1,t}\right)}, \frac{3\eta_{.,t}m^{1.5}G}{\sqrt{2}}\right\},
        \end{aligned}
    \end{equation}
    to complete the proof.
\end{proof}
\begin{lemma}
    \label{lem:gradient_norm_upper_bound_component1}
    In \radagrad, suppose that the Assumption~\ref{ass:comm_ass} and Assumption~\ref{ass:H_full_column_rank_ass} hold, $\Gamma\ge 1$, the perturbation $\delta_{i,t}$ is no decreasing and $\delta_{m,i} \le imG^2$, we have
    \begin{equation}
        \label{ineq:gradient_norm_upper_bound_component1}
        \begin{aligned}[b]
            & \left[\sum_{j=1}^m \grad f_{\sB_j^i}(\vx_j^i)\right]^\top\mG^{-\frac{1}{2}}_{m,i}\left[\sum_{j=1}^m \grad f_{\sB_j^i}(\vx_j^i)\right]\\
            \ge &\sqrt{\frac{\delta_{m,i}-\delta_{m,i-1}}{4G^2imd\Gamma}}\left[\sum_{j=1}^m \grad f_{\sB_j^i}(\vx_j^i)\right]^\top\left(\mG_{m,i}-\mG_{m,i-1}\right)^{-\frac{1}{2}}\left[\sum_{j=1}^m \grad f_{\sB_j^i}(\vx_j^i)\right]
        \end{aligned}
    \end{equation}
\end{lemma}
\begin{proof}
    To simplify notations in the following analysis, we set $\hat{\mH}_{m,t} \coloneqq \sum\limits_{\tau=1}^t \mH_{m,\tau}\mH_{m,\tau}^\top$.
    Then, we have
    \begin{equation}
        \label{ineq:sum_of_HHT_tr_upper_bound}
        \begin{aligned}[b]
            \left[tr\left(\hat{\mH}^{\frac{1}{2}}_{m,i}\right)\right]^2 = &\left[\sqrt{\lambda_1\left(\hat{\mH}_{m,i}\right)}+ \sqrt{\lambda_2\left(\hat{\mH}_{m,i}\right)}+\ldots+\sqrt{\lambda_d\left(\hat{\mH}_{m,i}\right)}\right]^2\\
            \le & \left[\lambda_1\left(\hat{\mH}_{m,i}\right) + \lambda_2\left(\hat{\mH}_{m,i}\right)+\ldots + \lambda_d\left(\hat{\mH}_{m,i}\right)\right]\cdot d = tr\left(\hat{\mH}_{m,i}\right) \cdot d \mathop{\le}^{\circled{1}} imG^2d
        \end{aligned}
    \end{equation}
    where $\circled{1}$ establishes because of the definition of matrix $\hat{\mH}_{m,i}$:
    \begin{equation}
        \label{ineq:trace_H_upper_bound}
        tr\left(\hat{\mH}_{m,i}\right) = \sum_{j=1}^{i}\sum_{k=1}^m \left\|\grad f_{\sB_k^j}(\vx_k^j)\right\|^2 \le imG^2.
    \end{equation}
    According to the definition of $\mG_{m,i}\in \mathbb{R}^{d\times d}$ in Eq.~\ref{eq:G_matrix_definition}, we have
    \begin{equation}
        \label{ineq:lambda_G_root_upper_bound}
        \begin{aligned}[b]
            \lambda_{max}\left(\mG^{\frac{1}{2}}_{m,i}\right) = &\lambda_{max}\left[\left(\hat{\mH}_{m,i}+\frac{\delta_{m,i}}{\Gamma}I\right)^{\frac{1}{2}}\right] \mathop{\le}^{\circled{1}} \lambda_{max}\left(\hat{\mH}_{m,i}^{\frac{1}{2}} + \sqrt{\frac{\delta_{m,i}}{\Gamma}}\cdot I\right)\\
            \mathop{\le}^{\circled{2}} &\lambda_{max}\left(\hat{\mH}_{m,i}^{\frac{1}{2}}\right) + \sqrt{\frac{\delta_{m,i}}{\Gamma}} \le tr\left(\hat{\mH}_{m,i}^{\frac{1}{2}}\right)+ \sqrt{\frac{\delta_{m,i}}{\Gamma}} \mathop{\le}^{\circled{3}} \sqrt{imd}G + \sqrt{\frac{im}{\Gamma}}G \mathop{\le}^{\circled{4}} 2\sqrt{imd}G,
        \end{aligned}
    \end{equation}
    where $\circled{1}$ follows from the fact $\left(\mM + c\cdot I\right)^{\frac{1}{2}}\preceq \mM^{\frac{1}{2}} + \sqrt{c}\cdot I, \forall A\succeq 0$, $\circled{2}$ follows from the triangle inequality, $\circled{3}$ follows from Eq.~\ref{ineq:trace_H_upper_bound} and $\delta_{m,i} \le imG^2$, $\circled{4}$ follows from the fact $\Gamma\ge 1$.
    Besides, when $\delta_{m,i}$ is increasing with $i$'s growth, i.e., $\delta_{m,i} - \delta_{m,i-1}\ge 0$,the matrix $\left(\mG_{m,i}-\mG_{m,i-1}\right)$ satisfies
    \begin{equation}
        \label{ineq:tr_root_psd_lower_bound_ano}
        \begin{aligned}[b]
            &\lambda_{min}\left[\left(\mG_{m,i}-\mG_{m,i-1}\right)^{\frac{1}{2}}\right] = \lambda_{min}\left[\left(\mH_{m,i}\mH^T_{m,i}+\frac{\delta_{m,i} - \delta_{m,i-1}}{\Gamma}\mI\right)^{\frac{1}{2}}\right]
            \ge \sqrt{\frac{\delta_{m,i} - \delta_{m,i-1}}{\Gamma}}.
        \end{aligned}
    \end{equation}
    Hence, if we set
    \begin{equation}
        \label{eq:beta_i_selection}
        \begin{split}
            \beta_{i} = \sqrt{\frac{\delta_{m,i}-\delta_{m,i-1}}{4G^2imd\Gamma}},
        \end{split}
    \end{equation}
    then, we obtain
    \begin{equation}
        \label{ineq:order_sufficient_condition_1_ano}
        \begin{split}
            \beta_{i}\lambda_{max}\left(\mG^{\frac{1}{2}}_{m,i}\right) \mathop{\le}^{\circled{1}}  \sqrt{\frac{\delta_{m,i} - \delta_{m,i-1}}{\Gamma}} \mathop{\le}^{\circled{2}} \lambda_{min}\left[\left(\mG_{m,i}-\mG_{m,i-1}\right)^{\frac{1}{2}}\right],
        \end{split}
    \end{equation}
    where $\circled{1}$ follows from Eq.~\ref{ineq:lambda_G_root_upper_bound} and $\circled{2}$ follows from Eq.~\ref{ineq:tr_root_psd_lower_bound_ano}.
    With the fact $\mG^{\frac{1}{2}}_{m,i}, \left(\mG_{m,i}-\mG_{m,i-1}\right)^{\frac{1}{2}}$ are positive-definite matrices, we have
    \begin{equation}
        \label{ineq:order_sufficient_condition_2}
        \begin{aligned}[b]
            & \beta_{i}\mG^{\frac{1}{2}}_{m,i} \preceq \left(\mG_{m,i}-\mG_{m,i-1}\right)^{\frac{1}{2}} \mathop{\Rightarrow}^{\circled{1}} \beta_{i}\left(\mG_{m,i}-\mG_{m,i-1}\right)^{-\frac{1}{2}} \preceq \mG^{-\frac{1}{2}}_{m,i}\\
            \Rightarrow & \left[\sum_{j=1}^m \grad f_{\sB_j^i}(\vx_j^i)\right]^\top\left[\beta_{i}\left(\mG_{m,i}-\mG_{m,i-1}\right)^{-\frac{1}{2}}\right]\left[\sum_{j=1}^m \grad f_{\sB_j^i}(\vx_j^i)\right]\\
            &\le \left[\sum_{j=1}^m \grad f_{\sB_j^i}(\vx_j^i)\right]^\top\mG^{-\frac{1}{2}}_{m,i}\left[\sum_{j=1}^m \grad f_{\sB_j^i}(\vx_j^i)\right],
        \end{aligned}
    \end{equation}
    where $\circled{1}$ follows from Lemma~\ref{lem:matrix_inverse_positive_definite_lemma}. Hence, we complete the proof.
\end{proof}
\begin{lemma}
    \label{lem:gradient_norm_upper_bound_component2}
    In \radagrad, suppose that the Assumption~\ref{ass:comm_ass} and Assumption~\ref{ass:H_full_column_rank_ass} hold, and the perturbation $\delta_{i,t}$ is no decreasing. 
    If $0 \le \delta_{m,i} - \delta_{m,i-1}\le m\lambda_{max}\left(\mH_{m,i}\mH_{m,i}^\top\right)$ and $\Gamma\ge m$ then we have
    \begin{equation}
        \label{ineq:gradient_norm_upper_bound_component2_con}
        \begin{split}
            \left[\sum_{j=1}^m \grad f_{\sB_j^i}(\vx_j^i)\right]^\top \left(\mG_{m,i}-\mG_{m,i-1}\right)^{-\frac{1}{2}}\left[\sum_{j=1}^m \grad f_{\sB_j^i}(\vx_j^i)\right] \ge \sqrt{\frac{m}{2c^2_{\kappa}}} \left\|\sum_{j=1}^m \grad f_{\sB_j^i}(\vx_j^i)\right\|
        \end{split}
    \end{equation}
\end{lemma}
\begin{proof}
    To simplify notations, we abbreviate $\mH_{m,i}$ as $\mH$ whose SVD can be formulated as
    \begin{equation}
        \label{eq:gradient_matrix_svd}
        \begin{split}
            \mH = \mU \mSigma\mV^\top,\quad \mU\in \sR^{d\times d},\mSigma\in \sR^{d\times m},\mV\in\sR^{m\times m}, 
        \end{split}
    \end{equation}
    where $\mU$ and $\mV$ are unitary matrices. Specifically, with Assumption~\ref{ass:H_full_column_rank_ass}, $\mSigma$ and $\mV$ can be written as
    \begin{equation}
        \label{eq:gradient_matrix_sigma}
        \begin{aligned}[b]
            & \mSigma = \left[
                    \begin{matrix}
                        \tilde{\mSigma} \\
                        0
                    \end{matrix}
                \right],\quad \tilde{\mSigma} = \mathop{diag}\left\{\tilde{\lambda}_1, \tilde{\lambda}_2,\ldots,\tilde{\lambda}_m\right\},\quad \mV =\left[
                \begin{matrix}
                    \vv_1 & \vv_2 & \ldots & \vv_m
                \end{matrix}
            \right],\quad \vv_i\in\sR^{m\times 1},\quad \forall i\in\sI_m.
        \end{aligned}
    \end{equation}
    Hence, we can reformulate $\left\|\sum_{j=1}^m \grad f_{\sB_j^i}(\vx_j^i)\right\|$ as
    \begin{equation}
        \label{eq:gradient_norm_refor}
        \begin{aligned}[b]
            & \left\|\sum_{j=1}^m \grad f_{\sB_j^i}(\vx_j^i)\right\|^2 = \left\|\mH_{m,i}\left[\sum_{j=1}^m \ve_j\right] \right\|^2 =\left[\sum_{j=1}^m \ve_j\right]^\top\mH^\top\mH\left[\sum_{j=1}^m \ve_j\right] = \left[\sum_{j=1}^m \ve_j\right]^\top\mV
            \tilde{\mSigma}^2\mV^\top\left[\sum_{j=1}^m \ve_j\right] \\
            =& \left[
                \begin{matrix}
                    \left[\sum_{j=1}^m \ve_j\right]^\top \vv_1 & \left[\sum_{j=1}^m \ve_j\right]^\top \vv_2 \ldots \left[\sum_{j=1}^m \ve_j\right]^\top \vv_m
                \end{matrix}    
            \right]\left[
                \begin{matrix}
                    \tilde{\lambda}_1^2 & 0 & \ldots & 0\\
                    0 & \tilde{\lambda}_2^2 & \ldots & 0\\
                    \ldots & \ldots & \ldots & \ldots\\
                    0 & 0 & \ldots & \tilde{\lambda}_m^2
                \end{matrix}    
            \right]\left[
                \begin{matrix}
                    \vv_1^\top \left[\sum_{j=1}^m \ve_j\right]\\
                    \vv_2^\top \left[\sum_{j=1}^m \ve_j\right]\\
                    \ldots \\
                    \vv_m^\top \left[\sum_{j=1}^m \ve_j\right]\\
                \end{matrix}    
            \right]\\
            =& \sum_{j=1}^{m}\tilde{\lambda}^2_j\left\|\left[\sum_{k=1}^m \ve_k\right]^T\vv_j\right\|^2 \mathop{=}^{\circled{1}} \sum\limits_{j=1}^m \tilde{\lambda}_j^2\gamma_j^2,
        \end{aligned}
    \end{equation}
    where $\ve_i$ denotes the $0$-$1$ vector whose $i$-th coordinate is $1$ while other coordinates are $0$s. 
    Besides, $\circled{1}$ in Eq.~\ref{eq:gradient_norm_refor} is established because we set
    \begin{equation}
        \label{eq:e_vector_decomposition}
        \begin{split}
            \sum_{k=1}^m \ve_k = \gamma_1\vv_1 + \gamma_2\vv_2 + \ldots + \gamma_m\vv_m,\quad \left\|\sum_{k=1}^m \ve_k\right\|^2 = \sum_{k=1}^m \gamma_k^2 = m.
        \end{split}
    \end{equation}
    with the full-rank property of matrix $\mV$.    
    In addtion, we have
    \begin{equation}
        \label{eq:gradient_norm_upper_refor}
        \begin{aligned}[b]
            & \left[\sum_{j=1}^m \grad f_{\sB_j^i}(\vx_j^i)\right]^\top \left(\mG_{m,i}-\mG_{m,i-1}\right)^{-\frac{1}{2}}\left[\sum_{j=1}^m \grad f_{\sB_j^i}(\vx_j^i)\right]\\
            = & \left[\sum_{j=1}^m \ve_j\right]^\top\mH^\top \left[\mH\mH^T + \frac{\delta_{m,i} - \delta_{m,i-1}}{\Gamma}\mI\right]^{-\frac{1}{2}}\mH\left[\sum_{j=1}^m \ve_j\right]\\
            = & \left[\sum_{j=1}^m \ve_j\right]^\top\mV\mSigma^\top\mU^\top\left[\mU\left(\mSigma\mSigma^\top + \frac{\delta_{m,i}-\delta_{m,i-1}}{\Gamma}\mI\right)\mU^\top\right]^{-\frac{1}{2}}\mU\mSigma\mV \left[\sum_{j=1}^m \ve_j\right]\\
            = & \left[\sum_{j=1}^m \ve_j\right]^\top\mV\mSigma^\top\left[
                \begin{matrix}
                    \left(\tilde{\mSigma}^2+\frac{\delta_{m,i}-\delta_{m,i-1}}{\Gamma}\mI\right)^{-\frac{1}{2}} & 0\\
                    0 & \left(\frac{\delta_{m,i}-\delta_{m,i-1}}{\Gamma}\mI\right)^{-\frac{1}{2}}
                \end{matrix}
            \right]\mSigma\mV^\top\left[\sum_{j=1}^m \ve_j\right]\\
            = & \sum_{j=1}^m\left[\frac{\tilde{\lambda}^2_j}{\sqrt{\tilde{\lambda}_j^2 + \frac{\delta_{m,i}-\delta_{m,i-1}}{\Gamma}}}\left\|\left[\sum_{k=1}^m \ve_k\right]^\top\vv_j\right\|^2\right] \mathop{=}^{\circled{1}} \sum_{j=1}^m\left[\frac{\tilde{\lambda}^2_j\gamma_j^2}{\sqrt{\tilde{\lambda}_j^2 + \frac{\delta_{m,i}-\delta_{m,i-1}}{\Gamma}}}\right]\\
            \ge & \frac{\tilde{\lambda}^2_m}{\sqrt{\tilde{\lambda}_1^2 + \frac{\delta_{m,i}-\delta_{m,i-1}}{\Gamma}}}\sum_{j=1}^m \gamma_j^2= \frac{\tilde{\lambda}^2_m m}{\sqrt{\tilde{\lambda}_1^2 + \frac{\delta_{m,i}-\delta_{m,i-1}}{\Gamma}}}.
        \end{aligned}
    \end{equation}
    where $\circled{1}$ follows from Eq.~\ref{eq:e_vector_decomposition}.
    Hence, we obtain
    \begin{equation}
        \label{eq:gradient_norm_upper_refor_final}
        \begin{aligned}[b]
            & \left(\left[\sum_{j=1}^m \grad f_{\sB_j^i}(\vx_j^i)\right]^\top \left(\mG_{m,i}-\mG_{m,i-1}\right)^{-\frac{1}{2}}\left[\sum_{j=1}^m \grad f_{\sB_j^i}(\vx_j^i)\right]\right)^2\\
            \ge & \frac{\tilde{\lambda}^2_m m}{\sqrt{\tilde{\lambda}_1^2 + \frac{\delta_{m,i}-\delta_{m,i-1}}{\Gamma}}}\sum_{j=1}^m\left[\frac{\tilde{\lambda}^2_j\gamma_j^2}{\sqrt{\tilde{\lambda}_j^2 + \frac{\delta_{m,i}-\delta_{m,i-1}}{\Gamma}}}\right]\ge  \frac{\tilde{\lambda}^2_m m}{\tilde{\lambda}_1^2 + \frac{\delta_{m,i}-\delta_{m,i-1}}{\Gamma}}\sum_{j=1}^{m}\tilde{\lambda}^2_j\gamma_j^2.
        \end{aligned}
    \end{equation}
    As a result, if $\Gamma\ge m$ is established then we have $\delta_{m,i} - \delta_{m,i-1}\le \Gamma\tilde{\lambda}_1^2$ (with the definition of $\tilde{\lambda}_1$) and
    \begin{equation}
        \label{ineq:gradient_norm_upper_bound_component2}
        \begin{aligned}[b]
            & \left[\sum_{j=1}^m \grad f_{\sB_j^i}(\vx_j^i)\right]^\top \left(\mG_{m,i}-\mG_{m,i-1}\right)^{-\frac{1}{2}}\left[\sum_{j=1}^m \grad f_{\sB_j^i}(\vx_j^i)\right]\\
            \ge & \sqrt{\frac{\tilde{\lambda}^2_m m}{\tilde{\lambda}_1^2 + \frac{\delta_{m,i}-\delta_{m,i-1}}{\Gamma}}\sum_{j=1}^{m}\tilde{\lambda}^2_j\gamma_j^2} \ge \frac{\tilde{\lambda}_m}{\tilde{\lambda}_1}\cdot \sqrt{\frac{m}{2}}\sqrt{\sum_{j=1}^{m}\tilde{\lambda}^2_j\gamma_j^2} = \sqrt{\frac{m}{2c^2_{\kappa}}} \left\|\sum_{j=1}^m \grad f_{\sB_j^i}(\vx_j^i)\right\|
        \end{aligned}
    \end{equation}
    with Assumption~\ref{ass:H_full_column_rank_ass} to complete the proof.
\end{proof}
\begin{theorem}
    \label{thm:adagrad_convergence_rate}
    In \radagrad, suppose that the Assumption~\ref{ass:comm_ass} and Assumption~\ref{ass:H_full_column_rank_ass} hold. 
    If hyper-parameters $\delta_{j,i}$, $\Gamma$ and the step size $\eta$ satisfy
    \begin{equation}
        \label{eq:final_step_size_selection_ac}
        \begin{aligned}[b]
            \delta_{j,i} = \sum_{p=1}^{i-1}\sum_{q=1}^m \left\|\grad f_{\sB_q^p}(\vx_q^p)\right\|^2 + \sum_{q=1}^j \left\|\grad f_{\sB_q^i}(\vx_q^i)\right\|^2, \forall\ i\in\sI_T, j\in\sI_m,\quad m\le \Gamma\le n,\quad \mathrm{and}\quad \eta \le \frac{c_\sigma^2}{16nLG}
        \end{aligned}
    \end{equation}
    we have
    \begin{equation}
        \label{ineq:gradient_norm_convergence_4}
        \begin{split}
            \mathbb{E}_t\left[\left\|\grad f(\vx_1^t)\right\|\right] \le \frac{C_0}{\eta \sqrt{T}} + \frac{C_1}{\sqrt{T}} + \frac{C_2\eta}{\sqrt{T}} + \frac{C_3\eta^2}{\sqrt{T}} + \frac{C_4\ln(T)}{\sqrt{T}} + \frac{C_5\eta\ln(T)}{\sqrt{T}},
        \end{split}
    \end{equation}
    where $C_0, C_1, \ldots, C_5$  are constants and defined in the proof.
\end{theorem}
\begin{proof}
    According to the definition of $\delta_{j,i}$, we have $\delta_{m,i} \le imG^2$ due to the gradient norm upper bound assumption, i.e., the forth item in Assumption~\ref{ass:comm_ass}.
    Besides, we have
    \begin{equation}
        \label{ineq:delta_properties}
        \begin{aligned}[b]
            \frac{c_\sigma^2m^2}{16n}\mathop{\le}^{\circled{1}} \sum_{j=1}^m \left\|\grad f_{\sB_j^i}(\vx_j^i)\right\|^2 = \delta_{m,i}-\delta_{m,i-1} = tr(\mH_{m,i}\mH_{m,i}^\top) \le m\cdot \lambda_{max}\left(\mH_{m,i}\mH_{m,i}^\top\right),
        \end{aligned}
    \end{equation}
    where $\circled{1}$ follows from Lemma~\ref{lem:G_trace_lower_bound} 
    Then, we have
    \begin{equation}
        \label{ineq:comb_gradient_norm_upper_bound_LHS_1}
        \begin{aligned}[b]
            &\frac{c_{\sigma}m}{8\sqrt{2}Gc_{\kappa}}\cdot \sqrt{\frac{1}{nd\Gamma}}\cdot\frac{1}{\sqrt{i}}\left\|\sum_{j=1}^m \grad f_{\sB_j^i}(\vx_j^i)\right\| = \frac{c_{\sigma}m}{4\sqrt{n}}\cdot \left(2G\sqrt{imd\Gamma}\right)^{-1}\cdot \sqrt{\frac{m}{2c^2_{\kappa}}}\left\|\sum_{j=1}^m \grad f_{\sB_j^i}(\vx_j^i)\right\| \\
            \mathop{\le}^{\circled{1}} & \frac{\sqrt{\delta_{m,i}-\delta_{m,i-1}}}{2G\sqrt{imd\Gamma}} \cdot \sqrt{\frac{m}{2c^2_{\kappa}}}\left\|\sum_{j=1}^m \grad f_{\sB_j^i}(\vx_j^i)\right\| = \sqrt{\frac{\delta_{m,i}-\delta_{m,i-1}}{4G^2imd\Gamma}}\cdot \sqrt{\frac{m}{2c^2_{\kappa}}}\left\|\sum_{j=1}^m \grad f_{\sB_j^i}(\vx_j^i)\right\| \\
            \mathop{\le}^{\circled{2}} &\sqrt{\frac{\delta_{m,i}-\delta_{m,i-1}}{4G^2imd\Gamma}} \left[\sum_{j=1}^m \grad f_{\sB_j^i}(\vx_j^i)\right]^\top \left(\mG_{m,i}-\mG_{m,i-1}\right)^{-\frac{1}{2}}\left[\sum_{j=1}^m \grad f_{\sB_j^i}(\vx_j^i)\right]\\
            \mathop{\le}^{\circled{3}} &\left(\sum_{j=1}^m \grad f_{\sB_j^i}(\vx_j^i)\right)^\top\mG_{m,i}^{-\frac{1}{2}}\left(\sum_{j=1}^m \grad f_{\sB_i^j}(\vx_i^j)\right),
        \end{aligned}
    \end{equation}
    where $\circled{1}$ is established due to Eq.~\ref{ineq:delta_properties}, $\circled{2}$ follows from Lemma~\ref{lem:gradient_norm_upper_bound_component2} and $\circled{3}$ follows from Lemma~\ref{lem:gradient_norm_upper_bound_component1}.
    Hence, we obtain
    \begin{equation}
        \label{ineq:comb_gradient_norm_convergence_1}
        \begin{aligned}[b]
            &\frac{c_{\sigma}}{48Gc_{\kappa}\sqrt{nd\Gamma}}\cdot\frac{\eta_{.,i}}{\sqrt{i}}\left\|\sum_{j=1}^m \grad f_{\sB_j^i}(\vx_j^i)\right\| \le \frac{\eta_{.,i}}{4m}\left(\sum_{j=1}^m \grad f_{\sB_j^i}(\vx_j^i)\right)^\top\mG_{m,i}^{-\frac{1}{2}}\left(\sum_{j=1}^m \grad f_{\sB_j^i}(\vx_j^i)\right)\\
            \mathop{\le}^{\circled{1}} &f(\vx_1^i) - f(\vx_{m+1}^i) + \frac{2\eta_{.,i}^3L^2G^2m^3}{3\lambda_{min}^{1.5}\left(\mG_{1,i}\right)} + \frac{3\eta_{.,i}m^{1.5}G^3}{\sqrt{2}\lambda_{min}\left(\mG_{1,i}\right)}\\
            \mathop{\le}^{\circled{2}} & f(\vx_1^i) - f(\vx_{m+1}^i) + \frac{128\eta^3_{.,i}L^2G^2n^{1.5}\Gamma^{1.5}}{3c^3_{\sigma}(i-1)^{1.5}} + \frac{24\sqrt{2}\eta_{.,i}n\Gamma G^3}{c_{\sigma}^2(i-1)m^{0.5}},
        \end{aligned}
    \end{equation}
    where $\circled{1}$ follows from Lemma~\ref{lem:sufficient_descent}, and $\circled{2}$ is established when $i\ge 2$ due to the fact 
    \begin{equation}
        \label{ineq:lambda_inequalities}
        \begin{split}
            &\lambda_{min}\left(\mG_{1,i}\right)\ge \lambda_{min}\left(\mG_{m,i-1}\right) \ge \frac{\delta_{m,i-1}}{\Gamma} \ge \frac{\sum_{j=1}^{i-1}\sum_{k=1}^m \left\|\grad f_{\sB_k^j}(\vx_k^j)\right\|^2}{\Gamma} \mathop{\ge}^{\circled{1}} \frac{(i-1)c_\sigma^2m^2}{16n\Gamma}
        \end{split}
    \end{equation}
    Inequality $\circled{1}$ in Eq.~\ref{ineq:lambda_inequalities} follows from Lemma~\ref{lem:G_trace_lower_bound}.
    It should be notice that when $i=1$, there is
    \begin{equation}
        \label{ineq:comb_gradient_norm_convergence_2}
        \begin{aligned}[b]
            \frac{c_{\sigma}}{48Gc_{\kappa}\sqrt{nd\Gamma}}\cdot \eta_{.,1}\left\|\sum_{j=1}^m \grad f_{\sB_j^1}(\vx_j^1)\right\| \mathop{\le}^{\circled{1}}  &f(\vx_1^1) - f(\vx_{m+1}^1) + \frac{2\eta^3_{.,t}L^2m^3}{3\sqrt{\lambda_{min}\left(\mG_{m,t}\right)}}+ \frac{3\eta_{.,t}m^{1.5}G}{\sqrt{2}}\\
            \le & f(\vx_1^1) - f(\vx_{m+1}^1) + \frac{8\eta_{.,1}^3L^2m^2n^{0.5}\Gamma^{0.5}}{3c_\sigma} + \frac{3\eta_{.,t}m^{1.5}G}{\sqrt{2}}
        \end{aligned}
    \end{equation}
    where $\circled{1}$ follows from Lemma~\ref{lem:sufficient_descent}.
    To achieve some stationary point through \radagrad, for each epoch, we have 
    \begin{equation}
        \label{ineq:gradient_norm_convergence_1}
        \begin{aligned}[b]
            & \frac{1}{\sqrt{i}}\left\|\sum_{j=1}^m \grad f_{\sB_j^i}(\vx_1^i)\right\|\le \frac{1}{\sqrt{i}}\left[\left\|\sum_{j=1}^m \grad f_{\sB_j^i}(\vx_j^i)\right\| + \left\|\sum_{j=1}^m\left(\grad f_{\sB_j^i}(\vx_j^i)-\grad f_{\sB_j^i}(\vx_1^i)\right)\right\|\right]\\
            \le & \frac{1}{\sqrt{i}}\left\|\sum_{j=1}^m \grad f_{\sB_j^i}(\vx_j^i)\right\| + \frac{1}{\sqrt{i}}\sum_{j=1}^m\left\|\grad f_{\sB_j^i}(\vx_j^i)-\grad f_{\sB_j^i}(\vx_1^i)\right\|\le \frac {1}{\sqrt{i}}\left\|\sum_{j=1}^m \grad f_{\sB_j^i}(\vx_j^i)\right\| + \frac{L}{\sqrt{i}} \sum_{j=1}^m\left\|\vx_j^i-\vx_1^i\right\|\\
            \mathop{\le}^{\circled{1}} & \frac{1}{\sqrt{i}}\left\|\sum_{j=1}^m \grad f_{\sB_j^i}(\vx_j^i)\right\| + \frac{L}{\sqrt{i}}\sum_{j=1}^m \left(\min\left\{\frac{\eta_{.,i}(j-1)G}{\sqrt{\lambda_{min}\left(\mG_{1,i}\right)}}, \eta_{.,i}(j-1)\right\}\right) \\
            \mathop{\le}^{\circled{2}} &\frac{1}{\sqrt{i}}\left\|\sum_{j=1}^m \grad f_{\sB_j^i}(\vx_j^i)\right\| + \frac{2LG\eta_{.,i}m n^{0.5}\Gamma^{0.5}}{c_{\sigma}(i-1)},
        \end{aligned}
    \end{equation}
    where $\circled{1}$ follows from Lemma~\ref{lem:bound_of_parameters} and $\circled{2}$ follows from Eq.~\ref{ineq:lambda_inequalities} when $i\ge 2$. 
    Notice that if $i=1$, we have
    \begin{equation}
        \label{ineq:ineq:gradient_norm_convergence_i1}
        \begin{aligned}[b]
            \left\|\sum_{j=1}^m \grad f_{\sB_j^1}(\vx_1^1)\right\| \le \left\|\sum_{j=1}^m \grad f_{\sB_j^1}(\vx_j^1)\right\| + L\eta_{.,i}m^2
        \end{aligned}
    \end{equation}
    Thus, combining Eq.~\ref{ineq:gradient_norm_convergence_1} with Eq.~\ref{ineq:comb_gradient_norm_convergence_1}, when $i\ge 2$, we obtain
    \begin{equation}
        \label{ineq:gradient_norm_convegence_2}
        \begin{aligned}[b]
            &\frac{\eta_{.,i}}{\sqrt{i}}\left\|\sum_{j=1}^m \grad f_{\sB_j^i}(\vx_1^i)\right\| \le \frac{\eta_{.,i}}{\sqrt{i}}\left\|\sum_{j=1}^m \grad f_{\sB_j^i}(\vx_j^i)\right\| + \frac{2LG\eta^2_{.,i}m n^{0.5}\Gamma^{0.5}}{c_{\sigma}(i-1)}\\
            \mathop{\le}^{\circled{1}} & \frac{48Gc_{\kappa}\sqrt{nd\Gamma}}{c_{\sigma}} \cdot \left[ f(\vx_1^i) - f(\vx_{m+1}^i) + \frac{128L^2G^2\eta^3_{.,i}n^{1.5}\Gamma^{1.5}}{3c^3_{\sigma}(i-1)^{1.5}} + \frac{24\sqrt{2}G^3\eta_{.,t}n\Gamma}{c_{\sigma}^2(i-1)m^{0.5}}\right]+ \frac{2LG\eta^2_{.,i}m n^{0.5}\Gamma^{0.5}}{c_{\sigma}(i-1)},
        \end{aligned}
    \end{equation}
    where $\circled{1}$ follows from Eq.~\ref{ineq:comb_gradient_norm_convergence_1}.
    Then, we set $\eta_{.,i} = \eta$ for all $1\le i\le T$. 
    Summing up Eq.~\ref{ineq:gradient_norm_convegence_2} for $1\le i\le T$ and dividing both sides by $m\eta$, we obtain
    \begin{equation}
        \label{ineq:gradient_norm_convergence_3}
        \begin{aligned}[b]
            & \sum_{i=1}^T \frac{1}{m\sqrt{i}}\left\|\sum_{j=1}^m \grad f_{\sB_j^i}(\vx_1^i)\right\| = \frac{1}{m}\left\|\sum_{j=1}^m \grad f_{\sB_j^1}(\vx_1^1)\right\| + \sum_{i=2}^T \frac{1}{m\sqrt{i}}\left\|\sum_{j=1}^m \grad f_{\sB_j^i}(\vx_1^i)\right\|\\
            \mathop{\le}^{\circled{1}} & \frac{1}{m} \left\|\sum_{j=1}^m \grad f_{\sB_j^1}(\vx_j^1)\right\| + L\eta m+ \sum_{i=2}^T \frac{1}{m\sqrt{i}}\left\|\sum_{j=1}^m \grad f_{\sB_j^i}(\vx_1^i)\right\|\\
            \mathop{\le}^{\circled{2}} &  \frac{48Gc_{\kappa}\sqrt{nd\Gamma}}{c_{\sigma}m\eta}\left[f(\vx_1^1) - f(\vx_{m+1}^1) + \frac{8\eta^3L^2m^2n^{0.5}\Gamma^{0.5}}{3c_\sigma} + \frac{3\eta m^{1.5}G}{\sqrt{2}}\right] + L\eta m + \sum_{i=2}^T \frac{1}{m\sqrt{i}}\left\|\sum_{j=1}^m \grad f_{\sB_j^i}(\vx_1^i)\right\|\\
            \mathop{\le}^{\circled{3}} & \frac{48Gc_{\kappa}\sqrt{nd\Gamma}}{c_{\sigma}m\eta}\left[f(\vx_1^1) - f(\vx_{m+1}^1) + \frac{8\eta^3L^2m^2n^{0.5}\Gamma^{0.5}}{3c_\sigma} + \frac{3\eta m^{1.5}G}{\sqrt{2}}\right] + L\eta m\\
            & +\frac{48Gc_{\kappa}\sqrt{nd\Gamma}}{c_{\sigma}m\eta}\left[f(\vx_1^2) - f(\vx_{m+1}^T) + \sum_{i=2}^{T}\left(\frac{128\eta^3 L^2G^2n^{1.5}\Gamma^{1.5}}{3c^3_{\sigma}(i-1)^{1.5}} + \frac{24\sqrt{2}\eta n\Gamma G^3}{c_{\sigma}^2(i-1)m^{0.5}}\right)\right] + \sum_{i=2}^T \left( \frac{2LG\eta n^{0.5}\Gamma^{0.5}}{c_{\sigma}(i-1)}\right)\\
            = & \frac{48Gc_{\kappa}\sqrt{nd\Gamma}}{c_{\sigma}m} \cdot \left[\frac{f(\vx_1^1) - f(\vx_{m+1}^T)}{\eta} + \frac{8\eta^2L^2m^2n^{0.5}\Gamma^{0.5}}{3c_\sigma} + \frac{3 m^{1.5}G}{\sqrt{2}} \right]+L\eta m\\
            & + \frac{48Gc_{\kappa}\sqrt{nd\Gamma}}{c_{\sigma}m} \cdot \sum_{i=2}^{T}\left(\frac{128\eta^2L^2G^2n^{1.5}\Gamma^{1.5}}{3c^3_{\sigma}(i-1)^{1.5}} + \frac{24\sqrt{2}G^3n\Gamma}{c_{\sigma}^2(i-1)m^{0.5}}\right) + \sum_{i=2}^T \left( \frac{2LG\eta n^{0.5}\Gamma^{0.5}}{c_{\sigma}(i-1)}\right)\\
            \le & \frac{48Gc_{\kappa}\sqrt{nd\Gamma}}{c_{\sigma}m} \cdot \left[\frac{f(\vx_1^1) - f(\vx_{m+1}^T)}{\eta} + \frac{8\eta^2L^2m^2n^{0.5}\Gamma^{0.5}}{3c_\sigma} + \frac{3 m^{1.5}G}{\sqrt{2}}\right]+L\eta m\\
            & +\frac{48Gc_{\kappa}\sqrt{nd\Gamma}}{c_{\sigma}m} \cdot \left[\frac{128\eta^2G^2n^{1.5}\Gamma^{1.5}}{c^3_{\sigma}} + \frac{48\sqrt{2}G^3n\Gamma\ln(T)}{c_{\sigma}^2m^{0.5}} \right] + \frac{4LG\eta n^{0.5}\Gamma^{0.5}\ln(T)}{c_{\sigma}},
        \end{aligned}
    \end{equation}
    where $\circled{1}$ follows from Eq.~\ref{ineq:ineq:gradient_norm_convergence_i1}, $\circled{2}$ follows from Eq.~\ref{ineq:comb_gradient_norm_convergence_2} and $\circled{3}$ follows from Eq.~\ref{ineq:gradient_norm_convegence_2}. With the following constants
    \begin{equation}
        \label{eq:constant_definition}
        \begin{aligned}[b]
            & C_0 = \frac{48Gc_{\kappa}\left[f(\vx_1^1) - f(\vx_{m+1}^T)\right]}{c_{\sigma}} \cdot n^{0.5}m^{-1}d^{0.5}\Gamma^{0.5},\quad C_1 = \frac{72\sqrt{2}G^2c_{\kappa}}{c_{\sigma}}\cdot n^{0.5}m^{0.5}d^{0.5}\Gamma^{0.5},\\
            & C_2 = L m,\quad C_3 = \frac{128G L^2 c_{\kappa}}{c^2_{\sigma}}\cdot nmd^{0.5}\Gamma+\frac{3\cdot 2^{11}G^3c_{\kappa}}{c_{\sigma}^{4}}\cdot n^2m^{-1}d^{0.5}\Gamma^{2},\\
            & C_4 = \frac{3^2\cdot 2^8 \sqrt{2}G^4c_{\kappa}}{c_{\sigma}^3} \cdot n^{1.5}m^{-1.5}d^{0.5}\Gamma^{1.5},\quad C_5 = \frac{4LG}{c_{\sigma}}\cdot n^{0.5} \Gamma^{0.5}d^{0.5},
        \end{aligned}
    \end{equation}
    we have
    \begin{equation*}
        \label{ineq:gradient_norm_convergence_4}
        \begin{split}
            \mathbb{E}_t\left[\left\|\grad f(\vx_1^t)\right\|\right] = \frac{\sum_{i=1}^T \frac{1}{\sqrt{i}}\left\|\sum_{j=1}^m \frac{1}{m} \grad f_{\sB_j^i}(\vx_1^i)\right\|}{\sum_{i=1}^T\frac{1}{\sqrt{i}}} \le \frac{C_0}{\eta \sqrt{T}} + \frac{C_1}{\sqrt{T}} + \frac{C_2\eta}{\sqrt{T}} + \frac{C_3\eta^2}{\sqrt{T}} + \frac{C_4\ln(T)}{\sqrt{T}} + \frac{C_5\eta\ln(T)}{\sqrt{T}},
        \end{split}
    \end{equation*}
    if we sample from $\sI_T$ with probability $\mathbb{P}\left[\rx=i\right] = \frac{1}{\sqrt{i}}$.
    It means we achieve some stationary points ($\left\|\grad f(\vx)\right\|\le \epsilon$) within $\tilde{O}(T^{-0.5})$ in expectation when we set the step size as $\eta = \Theta(1)$.
\end{proof}
\begin{corollary}
    Let $\left\{\vx_i^t\right\}$ be the sequence generated by \radagrad and $\vx_{out}$ be its output. 
    For given tolerance $\epsilon > 0$, under the same conditions as Theorem~\ref{thm:adagrad_convergence_rate}, if we choose the constant learning rate $\eta = \frac{c_{\sigma}^2}{16nLG}$, $\Gamma = m$ and the number of iteration in each epoch $m = n$, then to guarantee
    $$\mathbb{E}_t\left[\left\|\grad f(\vx_1^t)\right\|\right] = \frac{\sum_{i=1}^T \frac{1}{\sqrt{i}}\left\|\sum_{j=1}^m \frac{1}{m} \grad f_{\sB_j^i}(\vx_1^i)\right\|}{\sum_{i=1}^T\frac{1}{\sqrt{i}}} \le \epsilon,$$
    it requires nearly $T = \lfloor 36C_{max}n^3 d \epsilon^{-2}\rfloor$ outer iterations, where $C_{max}$ is set as
    \begin{equation}
        \label{eq:max_constant_definition}
        \begin{aligned}[b]
            C_{max} = \max\left\{\frac{384 LG^2 c_{\kappa}\left(f(\vx_1^1 - f^*)\right)}{c_{\sigma}^3}, \frac{72\sqrt{2}G^2c_{\kappa}}{c_{\sigma}},\frac{c^2_{\sigma}}{2G},\frac{24Gc_{\kappa}}{L^2},\frac{3^2\cdot 2^8 \sqrt{2}G^4c_{\kappa}}{c_{\sigma}^3}, \frac{c_{\sigma}}{4}\right\}
        \end{aligned}
    \end{equation}
    In expectation, the total number of gradient evaluation is nearly $\mathcal{T} = \Big\lfloor 36\left[1-\exp\left(-\frac{c_{\sigma}^{4}}{32G^4}\right)\right]^{-1}C_{max}n^4d\epsilon^{-2} \Big\rfloor$.
\end{corollary}
\begin{proof}
    According to Theorem~\ref{thm:adagrad_convergence_rate}, if we set $\eta = \frac{c_{\sigma}^2}{16nLG}$, $\Gamma = m$ and $m=n$, we obtain that
    \begin{equation}
        \label{ineq:exp_gradient_upper_bound}
        \begin{aligned}[b]
            \mathbb{E}_t\left[\left\|\grad f(\vx_1^t)\right\|\right] \le &\frac{384 LG^2 c_{\kappa}\left(f(\vx_1^1 - f^*)\right)}{c_{\sigma}^3}\cdot \sqrt{\frac{n^2d}{T}} + \frac{72\sqrt{2}G^2c_{\kappa}}{c_{\sigma}}\cdot \sqrt{\frac{n^3d}{T}} + \frac{c^2_{\sigma}}{16G}\cdot\frac{1}{\sqrt{T}}\\
            & + \left(\frac{c^2_{\sigma}}{2G}+\frac{24Gc_{\kappa}}{L^2}\right)\cdot \sqrt{\frac{n^2d}{T}} + \frac{3^2\cdot 2^8 \sqrt{2}G^4c_{\kappa}}{c_{\sigma}^3} \cdot \sqrt{\frac{n^3d}{T}}\ln(T)+ \frac{c_{\sigma}}{4}\cdot \sqrt{\frac{d}{T}}\ln(T)\\
            \le & 6C_{max} \cdot n^{1.5}\sqrt{d} \cdot \frac{\ln(T)}{\sqrt{T}}.
        \end{aligned}
    \end{equation}
    Hence, a sufficient condition for achieving \emph{FSP}s ($\mathbb{E}_t\left[\left\|\grad f(\vx_1^t)\right\|\right] \le \epsilon$) for the objective can be presented as
    \begin{equation}
        \label{ineq:sufficient_condition_for_convegrence}
        6C_{max} \cdot n^{1.5}\sqrt{d} \cdot \frac{\ln(T)}{\sqrt{T}}\le \epsilon \mathop{\Longleftrightarrow}^{\circled{1}} T\ge 36C_{max}n^3d\epsilon^{-2},
    \end{equation}
    where $\circled{1}$ is established when we ignore the $ln$ term.

    Besides, for the inner loops, we utilize the rejection sampling to provide a lower bound of $\delta_{m,t}$.
    According to Lemma~\ref{lem:G_trace_lower_bound}, we can notice that probability of success is at least $p \coloneqq 1-\exp\left(-\frac{m^2c^4_{\sigma}}{32n^2G^4}\right)$ in every trial (a Bernoulli distribution).
    Then, let $\rr$ be a random variable that indicates number of trials until success.
    The expectation of $\rr$ is 
    \begin{equation}
        \mathbb{E}\left[\rr\right] = \sum_{j=1}^{\infty} jp(1-p)^{j-1} = 1/p,\quad when\ p\in(0,1).
    \end{equation}
    As a result, it requires $\left[1-\exp\left(-\frac{c_{\sigma}^{4}}{32G^4}\right)\right]^{-1}n$ gradient evaluation for each epoch, and the total number of gradient evaluation is nearly $\mathcal{T} = \Big\lfloor 36\left[1-\exp\left(-\frac{c_{\sigma}^{4}}{32G^4}\right)\right]^{-1}C_{max}n^4d\epsilon^{-2} \Big\rfloor$ in expectation.
\end{proof}
\newpage
\section{The CNN Architecture of the Experiments}
Our model architecture is illustrated in Figure~\ref{fig:cnn_model}.
The first convolution layer consumes the input image and produce 6-channel output with a $5 \times 5$ convolution kernel. Then a $2 \times 2$ max-pooling layer is utilized, followed by another $5 \times 5$  convolution layer which produces 10-channel output. After two feed-forward layer with $10$ units, we predict the classification result using softmax.
\begin{figure*}[!tb]
    \centering
    \includegraphics[width=0.7\textwidth]{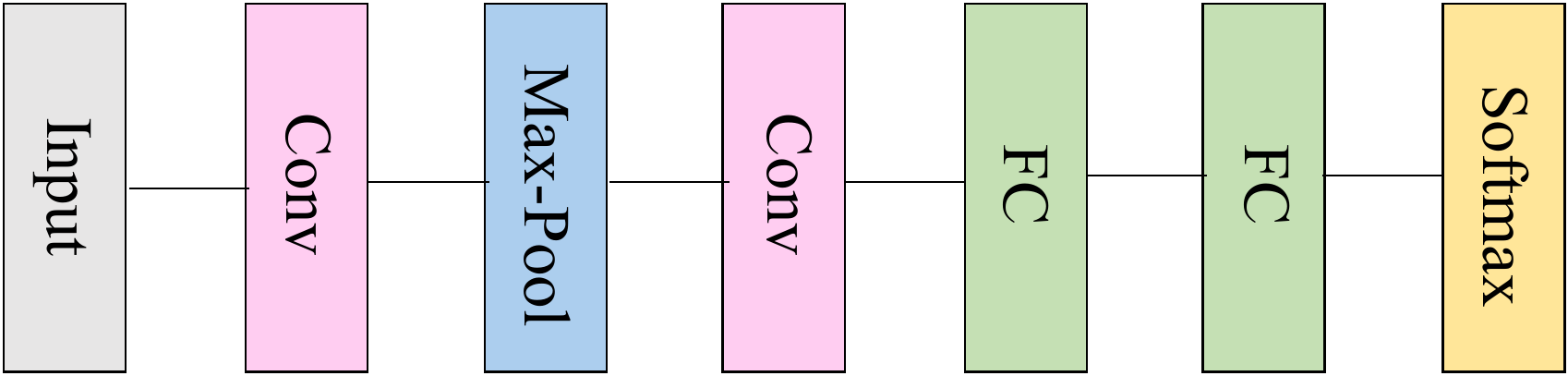}
    \caption{The architecture of CNN in our experiments}
    \label{fig:cnn_model}
\end{figure*}

\end{document}